\documentclass[a4paper,twoside]{article}

\usepackage[a4paper,left=3cm,right=3cm, top=3cm, bottom=3cm]{geometry}
\usepackage[latin1]{inputenc}
\usepackage[dvipsnames]{xcolor}
\usepackage{mathrsfs}
\usepackage{graphicx}
\usepackage{epstopdf}
\usepackage{amsmath}
\usepackage{amssymb}
\usepackage{amsthm}
\usepackage[hidelinks=true]{hyperref}
\usepackage{stmaryrd}
\usepackage{float}
\usepackage{bigints}
\usepackage{cite}
\usepackage{color}
\usepackage[abs]{overpic}
\usepackage[font=footnotesize,labelfont=bf]{caption}
\usepackage{cases}
\usepackage{tikz}
\usepackage{algorithm,algpseudocode}
\usepackage{rotating}
\usepackage{blkarray}
\usetikzlibrary{matrix,calc,arrows}
\usepackage[author={Lorenzo}]{pdfcomment}
\usepackage{verbatim}
\usepackage{graphicx}
\usepackage{subcaption}
\usepackage{algorithm}
\usepackage{pifont}
\usepackage{multirow}
\usepackage{tgpagella}
\usepackage{bm}

 \restylefloat{table}
 \theoremstyle{plain}
 \newtheorem{thm}{Theorem}[section]
 \newtheorem{cor}[thm]{Corollary}
 \newtheorem{lem}[thm]{Lemma}
 \newtheorem{prop}[thm]{Proposition}
 \theoremstyle{definition}

 \theoremstyle{remark}
 \newtheorem{remark}{Remark}

\newcommand{\p}{p}
\newcommand{\h}{h}
\newcommand{\E}{K}
\newcommand{\pE}{\p_\E}
\newcommand{\hE}{\h_\E}
\newcommand{\he}{\h_\e}
\newcommand{\Gcal}{\mathcal G}
\newcommand{\Gcalp}{\bm{\Gcal}_\p}
\newcommand{\Fcaln}{\mathcal{F}_n}

\newcommand{\mesh}{\mathcal{T}}
\newcommand{\xperp}{\bm{x}^{\bot}}
\newcommand{\Hcalp}{\bm{\mathcal{H}}_p}
\newcommand{\Hcalpmt}{\bm{\mathcal{H}}_{p-2}}
\newcommand{\Gcalpperp}{\bm{\Gcal}_p^\perp}
\newcommand{\Gcalpmt}{\bm{\Gcal} _{\p-2}}

\newcommand{\un}{\bm{u}_n}
\newcommand{\untilde}{\widetilde{\bm{u}}_n}

\newcommand{\zntilde}{\widetilde{\bm{z}}_n}
\newcommand{\vn}{\bm{v}_n}

\newcommand{\deltan}{\bm{\delta}_n}

\newcommand{\vnbar}{\overline{\bm{v}}_n}
\newcommand{\qnbar}{\overline{{q}}_n}
\newcommand{\vntilde}{\widetilde{\bm{v}}_n}
\newcommand{\wn}{{\bm{w}}_n}
\newcommand{\wnE}{{\bm{w}}_n^K}
\newcommand{\Vbf}{\bm{V}}
\newcommand{\Vn}{\bm{V}_n}
\newcommand{\Zn}{\bm{Z}_n}
\newcommand{\Vntilde}{\widetilde{\bm{V}}_n}
\newcommand{\Zntilde}{\widetilde{\bm{Z}}_n}
\newcommand{\Wn}{\bm{W}_n}
\newcommand{\VnE}{\Vn(\E)}
\newcommand{\VnEtilde}{\widetilde{\Vbf}_n(\E)}
\newcommand{\Qn}{Q_n}
\newcommand{\e}{e}
\newcommand{\EE}{\mathcal E^\E}
\newcommand{\q}{q}
\newcommand{\qalphaperp}{\bm{\q_{\alpha}}^\oplus}
\newcommand{\qpmtperp}{\bm{\q}_{\p-2}^\oplus}
\newcommand{\qpmo}{\q_{\p-1}}
\newcommand{\qboldtildepmt}{\widetilde{\bm{\q}} _{\p-2}}
\newcommand{\SE}{S^\E}
\newcommand{\SED}{\SE_D}

\newcommand{\alphabold}{\boldsymbol {\alpha}}
\newcommand{\qalpha}{\q_{\alpha}}
\newcommand{\palpha}{\p_{\alpha}}
\renewcommand{\a}{a}
\renewcommand{\b}{b}
\newcommand{\bn}{\b_n}
\newcommand{\aE}{\a^\E}
\newcommand{\bE}{\b^\E}
\newcommand{\ubf}{\bm{u}}
\newcommand{\vbf}{\bm{v}}
\newcommand{\PinablapE}{\bm{\Pi}^{\nabla,\E}_\p}
\newcommand{\Pinablap}{\bm{\Pi}^{\nabla}_\p}
\newcommand{\PizpmtE}{\bm{\Pi} ^{0,\E}_{\p-2}}
\newcommand{\Pizpmt}{\bm{\Pi} ^{0}_{\p-2}}
\newcommand{\PizpmoE}{\Pi^{0,\E}_{\p-1}}
\newcommand{\qboldp}{\bm{\q}_\p}

\newcommand{\qboldpmt}{\bm{\q}_{\p-2}}
\newcommand{\n}{\bm{n}}
\newcommand{\nE}{\n^\E}

\newcommand{\taun}{\mathcal T_n}

\newcommand{\Deltabold}{\boldsymbol \Delta}
\newcommand{\fbf}{\bm{f}}
\newcommand{\zerobold}{\boldsymbol 0}

\newcommand{\dalethn}{\daleth_n}

\newcommand{\PpmtEvec}{[\mathbb{P}_{p-2}(\E)]^2}

\newcommand{\PpEvec}{[\mathbb{P}_{p}(\E)]^2}
\newcommand{\xbf}{\bm{x}}
\newcommand{\xbfE}{\xbf_\E}
\newcommand{\Vcaln}{\mathcal V_n}
\newcommand{\En}{\mathcal E_n}
\newcommand{\VcalnI}{\Vcaln^I}
\newcommand{\EnI}{\En^I}
\newcommand{\VcalnB}{\Vcaln^B}
\newcommand{\EnB}{\En^B}
\newcommand{\qn}{q_n}
\newcommand{\sn}{s_n}
\newcommand{\an}{a_n}
\newcommand{\anE}{\an^\E}
\newcommand{\varphibold}{\boldsymbol{\varphi}}
\newcommand{\dofbf}{\textbf{dof}}
\newcommand{\Ncal}{\mathcal N}
\newcommand{\Etilde}{\widetilde \E}
\newcommand{\Tn}{n_{\mathrm{el}}}
\DeclareMathOperator{\card}{card}
\newcommand{\pbf}{\bm{p}}
\newcommand{\pbfEn}{\pbf^{\En}}
\newcommand{\Epsilonn}{n_{\mathrm{edge}}}
\newcommand{\NV}{N_V}

\newcommand{\cQ}{\mathcal{Q}}

\newcommand{\alphacoer}{\alpha_*(p)}
\newcommand{\alphacoerhat}{\widehat{\alpha}_*(p)}
\newcommand{\alphacont}{\alpha^*(p)}
\newcommand{\alphaconthat}{\widehat{\alpha}^*(p)}

\newcommand{\Dv}{\mathbf{Dv}^\E}
\newcommand{\Dvtilde}{\widetilde{\mathbf{Dv}}^\E}
\newcommand{\Id}{\operatorname{Id}}
\DeclareMathOperator{\TE}{T^{\E}_{\mathsf{StP}}}
\DeclareMathOperator{\TF}{T_{\mathsf{StP}}}

\newcommand{\sumE}{\sum_{\E\in \taun}}
\newcommand{\upi}{\bm{u}_\pi}
\renewcommand{\div}{\operatorname{div}}
\renewcommand{\wn}{\bm{w}_n}

\newcommand{\fc}{\mathfrak{c}}
\newcommand{\fci}{\fc_i}
\newcommand{\fC}{\mathfrak{C}}
\newcommand{\ugamma}{{\underline{\gamma}}}
\newcommand{\cK}{\mathcal{K}}
\newcommand{\alpham}{|\alphabold|}
\newcommand{\dalpha}{\partial^{\alphabold}}

\DeclareMathOperator{\dist}{dist}

\usepackage[most]{tcolorbox}

\begin{tiny}
\author{
\normalsize{
}}
\end{tiny}

\date{}

\title{{\textbf{$\p$- and $\h\p$- virtual elements for the Stokes problem}}}
\date{}
\author{{A. Chernov \thanks{Inst. f\"ur Mathematik,  Universit\"at Carl von Ossietsky, Oldenburg, Germany (alexey.chernov@uni-oldenburg.de)},\quad
C. Marcati \thanks{Seminar for Applied Mathematics - ETH Z\"urich, Z\"urich (carlo.marcati@sam.math.ethz.ch )}, \quad
L. Mascotto\thanks{Fakult\"at f\"ur Mathematik, Universit\"at Wien, 1090 Vienna, Austria (lorenzo.mascotto@univie.ac.at)}}}

\begin{document}
%%%%%%%%%%%%%%%%%%%%%%%%%%%%%%%%%%%%%
\maketitle
\begin{abstract}
\noindent
We analyse the $\p$- and~$\h\p$-versions of the virtual element method (VEM) for the the Stokes problem on a polygonal domain.
The key tool in the analysis is the existence of a bijection between Poisson-like and Stokes-like VE spaces for the velocities.
This allows us to re-interpret the standard VEM for Stokes as a VEM, where the test and trial discrete velocities are sought in Poisson-like VE spaces.
The upside of this fact is that we inherit from~\cite{hpVEMcorner} an explicit analysis of best interpolation results in VE spaces, as well as stabilization estimates that are explicit in terms of the degree of accuracy of the method.
We prove exponential convergence of the $hp$-VEM for Stokes problems with regular right-hand sides.
We corroborate the theoretical estimates with numerical tests for both the $p$- and $hp$-versions of the method.
\medskip


\end{abstract}

%%%%%%%%%%%%%%%%%%%%%%%%%%%%%%%%%%%%
\section{Introduction} \label{section:introduction}
%%%%%%%%%%%%%%%%%%%%%%%%%%%%%%%%%%%%
The virtual element method (VEM) is an increasingly popular tool in the approximation to solutions of fluido-static and dynamic problems in polygonal/polyhedral meshes.
In particular we recall: the very first paper on low-order VEM for Stokes~\cite{streamvirtualelementformulationstokesproblempolygonalmeshes};
its high-order conforming~\cite{BLV_StokesVEMdivergencefree} and nonconforming versions~\cite{CGM_nonconformingStokes, liu2017nonconforming};
conforming~\cite{NavierStokesVEM} and nonconforming VEM for the Navier-Stokes equation~\cite{nc_VEM_NavierStokes};
mixed VEM for the pseudo-stress-velocity formulation of the Stokes problem~\cite{Gatica-1}; mixed VEM for quasi-Newtonian flows~\cite{caceres2018mixed}; mixed VEM for the Navier-Stokes equation~\cite{gatica2018mixed};
other variants of the VEM for the Darcy problem~\cite{wang2019divergence, vacca2018h, caceres2017mixed};
analysis of the Stokes complex in the VEM framework~\cite{NavierStokes-complex-VEM, VEM-Stokes-complex-3D};
a stabilized VEM for the unsteady incompressible Navier-Stokes equations~\cite{irisarri2019stabilized};
implementation details~\cite{bricksVEM}.

Notwithstanding, all the above articles refer to the $\h$-version of the method
(i.e., when the convergence is achieved by refinement of the underlying mesh while keeping the order of the approximation fixed) and the convergence analysis
is performed assuming enough smoothness of the solutions to the problem under consideration.
This is not the case when the domain of the equation is polygonal/polyhedral.
In fact, even with smooth data, solutions are expected to have singularities at the corners of the domain; see, e.g., \cite{Guo2006, Marcati2019}.
More precisely, it can be proven that they belong to Kondrat'ev spaces,
i.e., weighted Sobolev spaces with weight given by a function of the distance from the corners of the domain; see definitions~\eqref{eq:cK} and~\eqref{eq:cKanalytic} below.

For this reason, employing $\h\p$ spaces arises as a natural technique in order to construct methods, which lead to an exponential decay of the error.
This approach has been investigated in a plethora of works, in the framework of conforming and nonconforming finite element methods.
We recall the following works, which relate to the $\h\p$ approximation of problems of Stokes and Navier-Stokes type:
$\h\p$ dG primal and mixed methods for the Stokes equation~\cite{schotzau2003exponential, schwab1999mixed}; 
mixed discontinuous Galerkin (dG) finite element methods for the Navier-Stokes equation~\cite{SMS20_888};
error indicator for the Stokes equation~\cite{bernardi2001error}; analysis of Stokes flows~\cite{gerdes1999hp};
mixed $\h\p$-dG methods for incompressible flows~\cite{schotzau2002mixed, schotzau2004mixed, schotzau2003stabilized} and their a posteriori version~\cite{paul2005energy};
spectral elements for Stokes eigenvalue problems~\cite{shan2017triangular}.

The main contribution of this paper is given by the development of the analysis of $\p$- and~$\h\p$-VEM for the approximation of solutions to the Stokes problem,
building upon the analysis for $\p$- and $\h\p$-VEM for the Poisson problem in~\cite{hpVEMbasic, hpVEMcorner}.
The key tool in the analysis is the proof of the existence of a bijection between Poisson-like \cite{VEMvolley} and Stokes-like \cite{BLV_StokesVEMdivergencefree} VE spaces for the velocities.
This allows us to re-interpret the standard VEM for Stokes~\cite{BLV_StokesVEMdivergencefree} as a VEM,
where the test and trial discrete velocities are sought in Poisson-like VE spaces.
The upside of this fact is that we inherit from~\cite{hpVEMcorner} an explicit analysis of best interpolation results in VE spaces, as well as stabilization estimates that are explicit in terms of the degree of accuracy of the method.

We prove that the $\h\p$-version of the method converges exponentially in terms of the cubic root of the number of degrees of freedom
when the right-hand side of the Stokes problem in a polygonal domain is analytic.
In addition, we also show that the $\p$-version of the method converges algebraically if the solution is sufficiently regular, and exponentially in terms of the degree of accuracy when the solution is analytic.

In the remainder of this section, we introduce some notation, the
continuous problem we are interested in, namely a Stokes problem in a two
dimensional polygonal domain, and
discuss the regularity of solutions to this kind of problems in polygonal
domains. Finally, we conclude this section by presenting the structure of the paper.
%%%%
\subsubsection*{Notation} 
%%%%
We employ the standard notation for Sobolev spaces~\cite{adamsfournier}. More precisely, given a domain~$D \subset \mathbb R^d$, $d=1,2$,
we denote the Sobolev space of integer order~$s\in \mathbb N$ by~$H^s(D)$.
We endow~$H^s$ with standard Sobolev inner products, seminorms and norms by
\[
(\cdot, \cdot)_{s,D} , \quad \quad \vert \cdot \vert_{s,D}, \quad \quad \Vert \cdot \Vert_{s,D}.
\]
Fractional Sobolev spaces can be defined via interpolation theory.
Moreover, we set~$\mathbb P_\p(D)$ as the space of polynomials of
total degree at most~$\p$ over the domain~$D$

As customary, given two positive quantities~$a$ and~$b$, we write~$a \lesssim b$
meaning that there exists a positive constant~$c$ independent of the discretization parameters such that~$a \le c \, b$.
Moreover, we write~$a\simeq b$ if~$a\lesssim b$ and~$b\lesssim a$ at once.

We write $\mathbb{N}_0 = \mathbb{N}\cup\{0\}$ and $\mathbb{R}^+ = \{x\in \mathbb{R}:x>0\}$.
%%%%
\subsubsection*{The continuous problem} 
%%%%
Let~$\Omega \subset \mathbb R^2$ be a polygonal domain with boundary~$\Gamma$ and~$\fbf \in [L^2(\Omega)]^2$.
We want to approximate the solution to the following problem: find~$\ubf$ and~$s$ such that
\begin{equation} \label{Stokes:strong}
\begin{cases}
-\Deltabold \ubf - \nabla s = \fbf & \text{in } \Omega\\
\div\ubf = 0 & \text{in } \Omega\\
\ubf = \zerobold & \text{on } \Gamma.\\
\end{cases}
\end{equation}
Define the spaces
\[
\Vbf :=[H^1_0(\Omega)]^2,\quad \quad  Q := L^2_0(\Omega) = \left\{q \in L^2(\Omega) ~:~ \int_\Omega q = 0 \right\},
\]
and the bilinear forms
\begin{equation} \label{continuous-global-bf}
\a(\ubf, \vbf) := (\nabla \ubf , \nabla \vbf)_{0,\Omega}, \quad\quad \b(\vbf, q) = (\div\vbf, q)_{0,\Omega} \quad \forall \ubf,\, \vbf \in \Vbf,\; \forall q \in Q.
\end{equation}
The weak formulation of problem~\eqref{Stokes:strong} reads
\begin{equation} \label{Stokes:weak}
\begin{cases}
\text{find } (\ubf, s) \in \Vbf \times Q \text{ such that} \\
\a(\ubf,\vbf) + \b(\vbf,s) = (\fbf, \vbf) & \forall \vbf \in \Vbf \\
\b(\ubf, q) = 0 & \forall q \in Q.\\
\end{cases}
\end{equation}
Problem~\eqref{Stokes:weak} is well-posed: see, e.g., \cite{BrezziFortin}.

%%%%
\subsubsection*{Regularity of the solution}
%%%%
The regularity of the solution~$(\ubf, s)$ to the Stokes problem~\eqref{Stokes:strong} in the polygonal domain~$\Omega$ depends on the shape of the domain.
In particular, even if the right-hand side~$\fbf$ is analytic, the corners of the domain give rise to corner singularities in the solution, which limit its regularity in the scale of classical Sobolev spaces.
In order to properly characterize the solution to the Stokes problem, we
resort to corner-weighted Sobolev spaces, of the kind firstly proposed in~\cite{Kondratev1967}.
\medskip 

\noindent Assume that the polygon~$\Omega$ has~$n_c\in\mathbb{N}$ corners, which we denote by~$\fC = \{\fc_i\in \mathbb{R}^2, i=1\dots, n_c\}$.
Set the amplitude of the internal angles at each corner~$\fci \in\fC$ as~$\phi_{\fci} \in (0, 2\pi)\setminus\{\pi\}$
and the Euclidean norm in~$\mathbb{R}^2$ by~$|\cdot|$.
Then, given the vector~$\ugamma = \{\gamma_{\fci}\in \mathbb{R}, \fci \in\fC\} \in \mathbb R^{n_c}$ and~$k\in \mathbb N_0$, define the weight function
\[
r^{k-\ugamma} (\xbf) : = \prod_{i=1}^{n_c} |\xbf - \fci|^{k-\gamma_{\fci}} \qquad \forall \xbf \in \Omega.
\]
For~$\ell\in \mathbb{N}_0$ and~$\ugamma\in \mathbb{R}^{n_c}$, introduce the seminorm and associated norm
\[
| v |^2_{\cK^\ell_\ugamma(\Omega)} := \sum_{\alphabold=(\alpha_1,\alpha_2)\in [\mathbb N_0]^2 , \,\alpham = \ell}\| r^{\alpham  -\ugamma}\dalpha v \|_{L^2(\Omega)}^2, \qquad \| v \|^2_{\cK^\ell_\ugamma(\Omega)} := \sum_{k=0}^\ell| v |_{\cK^k_\ugamma(\Omega)}^2,
\]
where we use the notation~$\dalpha = \partial^{\alpha_1}_{x_1}\partial^{\alpha_2}_{x_2}$.
We define the homogeneous Kondrat'ev space as
\begin{equation} \label{eq:cK}
\cK^{\ell}_\ugamma(\Omega) := \left\{ v\in L^2(\Omega) : \|v\|_{\cK^\ell_\ugamma(\Omega)}<\infty \right\}.
\end{equation}
Furthermore, we introduce the class of weighted analytic functions
\begin{equation} \label{eq:cKanalytic}
\cK^{\varpi}_\ugamma(\Omega) := \left\{ v\in \bigcap_{\ell\in \mathbb{N}_0}\cK^{\ell}_\ugamma(\Omega) : \exists A\in \mathbb{R}\text{ such that }|v|_{\cK^\ell_\ugamma(\Omega)}\leq A^{\ell+1}\ell! \,\forall\ell\in\mathbb{N}_0 \right\}.
\end{equation}
\medskip 

\noindent For each vertex~$\fci \in\fC$ of~$\Omega$, $\lambda_{\fci}$ denotes the smallest positive solution to the following equation:
\begin{equation}
  \label{eq:sing-exp}
\left(\sin(\lambda_{\fci} \phi_{\fci})\right)^2 = \lambda_{\fci}^2 \left(\sin\phi_{\fci}  \right)^2.
\end{equation}
Observe that, for all~$\phi_{\fci} \in (0, 2\pi)\setminus\{\pi\}$, we have~$\lambda_{\fci}>1/2$.
Furthermore, for all~$0 < \phi_c < \pi$, i.e., in presence of convex corners, we have~$\lambda_{\fci} = 1$.

The following result is a finite regularity shift result in weighted Sobolev spaces for solutions to the Stokes problem;
see~\cite[Theorem 5.7]{Guo2006} and~\cite[Section 5]{Kozlov2001}; see also~\cite[Proposition 1.8]{Marcati2019} for the case of homogeneous spaces.
\begin{thm} \label{theorem:regularity}
Let~$\ell\in \mathbb{N}_0$ and ~$\ugamma$ be such that~$0<\gamma_{\fci} - 1 < \lambda_{\fci}$ for all~${\fci} \in\fC$.
Assume that~$\fbf\in\left[\cK^{\ell}_{\ugamma-2}(\Omega)  \right]^2$ and let~$(\ubf,s)\in \Vbf\times Q$ be the (unique) solution to~\eqref{Stokes:strong} with right-hand side~$\fbf$.
Then, there exists~$C>0$ such that
\begin{equation}    \label{eq:reg-finite}
\| \ubf \|_{\cK^{\ell+2}_\ugamma(\Omega)}   + \|s\|_{\cK^{\ell+1}_{\ugamma-1}(\Omega)} \leq C \|\fbf\|_{\cK^\ell_{\ugamma-2}(\Omega)}.
\end{equation}
\end{thm}
Furthermore, if the right-hand side belongs to analytic weighted spaces, then also the solution to the Stokes problem belongs to the same spaces, as stated in the following result; see~\cite[Theorem 5.7]{Guo2006}.
\begin{thm} \label{theorem:regularity2}
Let~$\ugamma$ be such that~$0<\gamma_\fc - 1 <\lambda_\fc$ for all~$\fc\in\fC$.
Let~$\fbf\in \left[  \cK^\varpi_{\ugamma-2}(\Omega) \right]^2$ and~$(\ubf, s)\in \Vbf\times Q$ be the solution to~\eqref{Stokes:strong} with right-hand side $\fbf$.
Then~$\ubf \in \left[ \cK^\varpi_{\ugamma}(\Omega) \right]^2$ and~$s\in \cK^\varpi_{\ugamma-1}(\Omega) $.
\end{thm}

\subsubsection*{Structure of the paper.}
In Section~\ref{section:VEM}, we construct the VEM for the approximation of solutions to problem~\eqref{Stokes:weak}.
Differently from the standard approach of~\cite{BLV_StokesVEMdivergencefree},
we show that the VEM for the Stokes equation can be re-interpreted as a VEM where the velocity space is Poisson-like~\cite{VEMvolley}.
Section~\ref{section:a-priori} is concerned with the derivation of a priori estimates on velocities and pressures.
Among the key points here, we prove the validity of the inf-sup condition and stabilization bounds, which are explicit in terms of the degree of accuracy of the method.
The exponential convergence for the $\p$- and $\h\p$-versions of the method are theoretically proven in Section~\ref{section:convergence}, and numerically validated in Section~\ref{section:numerical-experiments}.
We draw some conclusions in Section~\ref{section:conclusions}.

%%%%%%%%%%%%%%%%%%%%%%%%%%%%%%%%%%%%
\section{Meshes and the virtual element method} \label{section:VEM}
%%%%%%%%%%%%%%%%%%%%%%%%%%%%%%%%%%%%
In this section, we present the virtual element method for the approximation of solutions to~\eqref{Stokes:weak}.
More precisely, we begin by introducing sequences of polygonal meshes partitioning the domain~$\Omega$ and their properties in Section~\ref{subsection:meshes}.
Next, in Section~\ref{subsection:VES-Stokes}, we recall the virtual element spaces introduced in~\cite{BLV_StokesVEMdivergencefree},
whereas, in Section~\ref{subsection:discrete-forms}, we construct computable bilinear forms and exhibit the method.
We devote, then, Section~\ref{subsection:VEM-Poisson} to recalling the standard virtual element method from~\cite{VEMvolley}.
Indeed, we show that the virtual element method for the Stokes equation can be re-interpreted as a method, where the velocity is sought, in Poisson-like virtual element spaces.
This fact will play an important role in the analysis presented in Section~\ref{section:a-priori} below.

%%%%%%%%%%%%%%%%%%%%%%%%
\subsection{Meshes} \label{subsection:meshes}
%%%%%%%%%%%%%%%%%%%%%%%%
Here, we introduce the polygonal meshes upon which we will construct the virtual element method.
Specifically, we consider sequences~$\{\taun\}_{n\in \mathbb N}$ of meshes which partition the domain~$\Omega$ into conforming, nonoverlapping polygons.
Fix~$n \in \mathbb N$, i.e., fix one of the meshes in the sequence.
We denote the set of vertices and edges in~$\taun$ by~$\Vcaln$ and~$\En$, respectively.
Next, fix~$\E \in \taun$. We denote its diameter and centroid by~$\hE$ and~$\xbfE$, respectively. Moreover, $\EE$ represents its set of edges.
We define~$\h:= \max_{\E \in \taun} \hE$.

The set of vertices~$\Vcaln$ and edges~$\En$ can be decomposed into internal and boundary, i.e., contained in~$\Gamma = \partial \Omega$, ones.
We write~$\VcalnI$, $\VcalnB$, $\EnI$, and~$\EnB$, respectively.
We denote the length of each edge $e\in \En$ by~$\he$.
\medskip

We state the following assumptions on the sequence of meshes: for all~$n\in \mathbb N$, there exists~$\gamma \in (0, 1)$ such that
\begin{itemize}
\item[(\textbf{A0-$\p$})] the mesh~$\taun$ is quasi-uniform, i.e., for all~$\E_1$ and~$\E_2 \in \taun$, there holds $\gamma \h_{\E_1} \le \h_{\E_2} \le \gamma^{-1} \h_{\E_1}$;
\item[(\textbf{A0-$\h\p$})] the mesh~$\taun$ is locally quasi-uniform, i.e., for all
  neighbouring~$\E_1$ and~$\E_2 \in \taun$, there holds $\gamma \h_{\E_1} \le \h_{\E_2} \le \gamma^{-1} \h_{\E_1}$;
\item[(\textbf{A1})] for all~$\E \in \taun$, $\E$ is star-shaped with respect to a ball with radius larger than or equal to~$\gamma \hE$;
\item[(\textbf{A2})] for all~$\E \in \taun$ and for all~$\e \in \EE$, there holds~$\hE \le \gamma \he$.
\end{itemize}
The assumptions~(\textbf{A1}) and~(\textbf{A2}) will be used throughout the whole paper.
Instead, assumptions~(\textbf{A0-$\p$}) and~(\textbf{A0-$\h\p$}) will be considered when dealing with the $\p$- and $\h\p$-version of the method, respectively.

For the sake of exposition, we construct the method for uniform~$\p$ only, and postpone to Section~\ref{subsection:hpVE-spaces} the variable degree case.

\begin{remark}  \label{remark:weird-geometries}
The forthcoming analysis can be also extended to more general geometries; see, e.g., \cite{brennerVEMsmall, cao2018anisotropic, beiraolovadinarusso_stabilityVEM}.
For the sake of clarity, we stick to the setting detailed above.
\end{remark}

We denote the space of piecewise discontinuous polynomials of degree~$\p \in \mathbb N$ over~$\taun$ by~$\mathbb P_\p(\taun)$.

%%%%%%%%%%%%%%%%%%%%%%%%
\subsection{The Stokes virtual element spaces} \label{subsection:VES-Stokes}
%%%%%%%%%%%%%%%%%%%%%%%%
Here, we recall from~\cite{BLV_StokesVEMdivergencefree} the virtual element spaces which we will use in the discretization of the Stokes problem~\eqref{Stokes:weak}.
Henceforth, $\p \in \mathbb N$ denotes the degree of accuracy
of the method.
Given~$\E \in \taun$, set
\[
\Gcalp (\E) : = \nabla (\mathbb P_{\p+1}(\E)) \subset \PpEvec
\]
and introduce the subspace~$\Hcalp (\E)\subset \PpEvec$ such that
\begin{equation}
\label{eq:Hcaldef}
\PpEvec = \Gcalp(\E) \oplus \Hcalp (\E).
\end{equation}
In~\cite{BLV_StokesVEMdivergencefree}, $\Hcalp(\E)$ is chosen as the~$L^2(\E)$-orthogonal complement in~$\PpEvec$ of~$\Gcalp(\E)$, denoted~$\Gcalpperp(\E)$.
In practical computations, see~\cite{bricksVEM}, a convenient choice is provided by the space
\[
\xperp \mathbb{P}_{\p-1}(\E), \qquad \xperp =
  \begin{pmatrix}
    -y \\ x
  \end{pmatrix}.
\]

In what follows, we do not impose orthogonality in \eqref{eq:Hcaldef}, but only require that~$\Hcalp(\E)$ is such that~\eqref{eq:Hcaldef} is a direct sum.

Recall that~$\EE$ denotes the set of edges of the element~$\E$ and introduce
\[
\mathbf{\mathbb B}_{\p}(\partial \E) := \left \{ \vn \in \mathcal C^0(\partial \E) \mid \vn{}_{|\e} \in [\mathbb P_{\p}(\e)]^2 \; \forall \e \in \EE  \right \}.
\]
Define the local bilinear forms
\[
\aE(\ubf, \vbf) := (\nabla \ubf, \nabla \vbf)_{0,\E} ,\quad \quad \bE( \vbf, \q) := (\div\vbf , \q)_{0,\E} \quad \quad \forall \ubf,\, \vbf \in [H^1(\E)]^2,\; \q \in L^2(\E).
\]
Consider the following local Stokes problem: Given~$\qpmtperp \in \Hcalp(\E)$ and~$\qpmo \in \mathbb P_{\p-1}(\E) / \mathbb R$,
\begin{equation} \label{local:problem}
\begin{cases}
\text{find } (\vn,s) \in H^1(\E) \times L^2(\E),& \vn{}_{|\partial \E} \in \mathbf{\mathbb B}_{\p}(\partial \E)  \text{ such that}\\
-\Delta \vn - \nabla s = \qpmtperp 	& \text{in } \E\\
\div \vn = \qpmo    					&\text{in } \E.\\
\end{cases}
\end{equation}
Set the local Stokes-like virtual element space for the velocity as follows:
\[
\VnE := \left\{  \vn \in [H^1(\E)]^2 \mid \vn \text{ solves a problem of the form~\eqref{local:problem}}   \right\}.
\]
We introduce the following linear functionals on~$\VnE$: given~$\vn \in \VnE$, define
\begin{itemize}
\item $\Dv_1(\vn)$: the point values at the vertices of~$\E$;
\item $\Dv_2(\vn)$: the point values at the~$\p-1$ Gau\ss-Lobatto points on each edge~$\e \in \EE$;
\item given~$\{ \qalphaperp \}$ a basis of~$\Hcalp$, the ``complementary'' moments 
\begin{equation} \label{orthogonal:moment}
\Dv_3(\vn)_\alpha = \frac{1}{\vert \E \vert} \int_\E \vn \cdot \qalphaperp;
\end{equation}
\item given~$\{ \qalpha \}_{\alpha=1}^{\p-1}$ a basis of~$\mathbb P_{\p-1}(\E) /\mathbb R$, the ``divergence'' moments
\begin{equation} \label{divergence:moment}
\Dv_4(\vn)_\alpha = \frac{\hE}{\vert \E \vert} \int_\E \div(\vn) \qalpha.
\end{equation}
\end{itemize}
\begin{lem} \label{lemma:standard-dofs-Stokes}
The above linear functionals are a set of degrees of freedom for~$\VnE$.
\end{lem}
\begin{proof}
See~\cite[Proposition~3.2]{BLV_StokesVEMdivergencefree}.
\end{proof}
We define the~$H^1$-conforming global Stokes-like velocity space as follows:
\begin{equation} \label{VEspace:velocity:global}
\Vn := \{ \vn \in [H^1_0(\Omega)]^2 : \vn {}_{|\E} \in \VnE \; \text{ for all } \E \in \taun       \}.
\end{equation}
We endow this space with the set of degrees of freedom, which is obtained by a standard $H^1$-conforming dof coupling of the local ones. 
\medskip

The above degrees of freedom allow us to compute two projection operators; see~\cite[Sections~$3.2$ and~$3.3$]{BLV_StokesVEMdivergencefree}.
The first one is the~$H^1$ projector~$\PinablapE : [H^1(\E)]^2 \rightarrow [\mathbb P_\p(\E)]^2$ defined as
\begin{equation} \label{Pinabla:definition}
\begin{cases}
\aE(\qboldp , \vn - \PinablapE \vn) = 0 \quad \quad \forall \qboldp \in [\mathbb P_{\p}(\E)]^2\\
\int_{\partial \E} \vn - \PinablapE \vn = \mathbf 0.\\
\end{cases}
\end{equation}
We define the global projector~$\Pinablap :  [H^1(\taun)]^2\to [\mathbb{P}_{\p}(\taun)]^2$ so that, for all~$\vbf\in [H^1(\taun)]^2$, 
\[
 \left( \Pinablap \vbf   \right) _{|\E} = \PinablapE (\vbf_{|\E}) \qquad \forall \E\in \taun.
\]
Furthermore, we can compute the~$L^2$ projector~$\PizpmtE : \VnE \rightarrow [\mathbb P_{\p-2} (\E)]^2$ defined as
\begin{equation} \label{L2projector:definition}
(\qboldpmt , \vn - \PizpmtE \vn )_{0,\E} = 0 \quad\quad \forall \qboldpmt \in [\mathbb P_{\p-2}(\E)]^2.
\end{equation}
These two operators are instrumental in the design of the virtual element methods; see Section~\ref{subsection:discrete-forms} below.

For future convenience, introduce the broken Sobolev space
\[
H^1(\taun) := \left\{ \vbf \in [L^2(\Omega)]^2: \vbf{}_{| \E }   \in [H^1(\E)]^2\; \forall \E\in \taun \right\},
\]
and associate with it the broken Sobolev seminorm and norm
\[
| \vbf |_{1, \taun}^2 := \sum_{\E\in\taun} \|\nabla \vbf\|^2_{0, \E}\qquad  \| \vbf \|_{1, \taun}^2 := \| \vbf\|^2_{0, \Omega}  + |\vbf|^2_{1, \taun}.
\]
Finally, set the pressure space as
\begin{equation} \label{eq:Qn}
\Qn : = \left\{ \qn \in L^2_0(\Omega): \qn{}_{|\E} \in \mathbb{P}_{\p-1}(\E) \; \text{ for all } \E\in \taun\right\}.
\end{equation}

%%%%%%%%%%%%%%%
\subsection{The virtual element method} \label{subsection:discrete-forms}
%%%%%%%%%%%%%%%
Here, we design computable discrete bilinear forms and right-hand side and introduce the virtual element method for the approximation of solutions to the Stokes problem~\eqref{Stokes:weak}.
\subsubsection*{Discrete bilinear forms.}
We introduce the elementwise discrete bilinear form~$a^{\E}_n$ given by
\begin{equation} \label{local:discrete:BF}
 \anE(\un, \vn) := \aE (\PinablapE \un, \PinablapE\vn) + \SE ((\Id - \PinablapE)\un,(\Id - \PinablapE)\vn ) \quad \quad \forall  \un, \vn\in \VnE,
\end{equation}
where, for all~$\E \in \taun$, $\SE: H^1(\E) \times H^1(\E) \to  \mathbb{R}$ is a computable local stabilizing bilinear form, which is computable from the degrees of freedom introduced in Section~\ref{subsection:VES-Stokes}.
We postpone the discussion about further properties of the stabilizing bilinear forms~$\SE$ to Section~\ref{subsection:stabilization} below.
The global discrete bilinear form reads
\[
\an(\un, \vn) = \sum_{\E\in \taun} a^{\E}_n(\un{}_{|\E}, \vn{}_{|\E}) \qquad \forall\un,\vn\in\Vn.
\]
As for the discretization of the bilinear form~$\b(\cdot,\cdot)$ in~\eqref{continuous-global-bf},
we observe that the divergence of functions in the space~$\VnE$ is polynomial and can be expressed in closed form in terms of their degrees of freedom. 
Therefore, no approximation is necessary for the second bilinear form and we  define 
\[
\bn(\vn,\qn) := \b(\vn,\qn) \quad \quad \forall \vn \in \Vn, \, \forall \qn \in \Qn.
\]

\subsubsection*{Discrete right-hand side.}
Define the global piecewise~$L^2$ projector~$\Pizpmt$ as follows: Given~$\fbf \in [L^2(\Omega)]^2$,
\[
 \left( \Pizpmt \fbf \right)_{|\E}= \PizpmtE(\fbf_{|\E}) \qquad \forall \E\in \taun.
\]

\subsubsection*{The virtual element method.}
The virtual element method for the Stokes problem~\eqref{Stokes:weak} reads as follows:
\begin{equation} \label{Stokes:VEM}
\begin{cases}
\text{find } (\un, \sn) \in \Vn \times \Qn \text{ such that} \\
\an(\un,\vn) + \b(\vn, \sn) = (\Pizpmt \fbf, \vn)_{0,\Omega} 	& \forall \vn \in \Vn \\
\b(\un, \qn) = 0 														& \forall \qn \in \Qn.
\end{cases}
\end{equation}

%%%%%%%%%%%%%%%%%
\subsection{An equivalent formulation in Poisson-like virtual element spaces} \label{subsection:VEM-Poisson}
%%%%%%%%%%%%%%%%%
We recall the vector Poisson-like virtual element space, see~\cite{VEMvolley},
for this will allow us to reinterpret method~\eqref{Stokes:VEM} in a way that is more convenient for the sake of the analysis in Section~\ref{section:a-priori} below.
Given~$\E \in \taun$, set
\[
\VnEtilde := \{  \vntilde \in [H^1(\E)]^2  ~:~ \vntilde{}_{|\partial \E} \in \mathbf{\mathbb B}_{\p}(\partial \E) \text{ and } \Delta\vntilde \in \mathbb [\mathbb{P}_{\p-2}(\E)]^2   \}.
\]
The global~$H^1$ standard Poisson-like virtual element space reads
\begin{equation} \label{Poisson-like:VEM}
\Vntilde = \left\{ \vntilde \in [H^1(\E)]^2 ~:~ \vntilde {}_{|\E} \in \VnEtilde \; \text{ for all } \E \in \taun       \right\}.
\end{equation}
The operators~$\Dv_i$, $i=1, \dots, 4$ introduced in Section~\ref{subsection:VES-Stokes} are unisolvent degrees of freedom for both~$\VnE$ and~$\VnEtilde$, as stated in the following lemma,
where we also prove that such degrees of freedom identify a bijection between the two virtual element spaces.
\begin{lem}  \label{lemma:bijection}
For all~$\E\in \taun$, there exists a Stokes-to-Poisson bijection~$\TE: \VnE \to \VnEtilde$ such that
\begin{equation} \label{equal:dofs}
\Dv_i(\vn) = \Dv_i(\TE\vn), \qquad i=1, 2, 3, 4, \quad \quad \forall \vn\in \VnE.
\end{equation}
\end{lem}
\begin{proof}
Given~$\E \in \taun$, introduce the following auxiliary set of degrees of freedom: given~$\vntilde\in \VnEtilde$,
 \begin{itemize}
\item $\Dvtilde_1(\vntilde)$: the point values at the vertices if~$\E$;
\item $\Dvtilde_2(\vntilde)$: the point values at the~$\p-1$ Gau\ss-Lobatto points on each edge~$\e \in \EE$;
\item given~$\{ \qalphaperp \}$ the basis of~$\Hcalpmt$ used in~\eqref{orthogonal:moment}, the moments
\begin{equation} \label{orthogonal:momenttilde}
\Dvtilde_3(\vntilde)_\alpha = \frac{1}{\vert \E \vert}\int_\E \vn \cdot \qalphaperp;
\end{equation}
\item given~$\{ \palpha \}_{\alpha=1}^{\p-1}$ a basis of~$\Gcalpmt$ such that~$\palpha =\nabla \qalpha$, with~$\qalpha$ defined in~\eqref{divergence:moment}, the moments 
\begin{equation} \label{divergence:momenttilde}
\Dvtilde_4(\vntilde)_\alpha = \frac{\hE}{\vert \E \vert} \int_\E \vn\cdot \palpha.
\end{equation}
\end{itemize}
Since~$[\mathbb{P}_{p-2}(K)]^2 = \Gcalpmt \oplus \Hcalpmt$, this is indeed a set of degrees of freedom; see~\cite[Proposition~4.1]{VEMvolley}.
Furthermore, $\Dvtilde_i = \Dv_i$ for~$i=1,2,3$.

For any~$\vn\in \VnE$, introduce~$\vntilde = \TE\vn \in \VnEtilde$ as described below. First, we require  $\vntilde{}_{|\partial \E} = \vn{}_{|\partial \E}$.
In other words, fix~$\Dv_i(\vntilde) =\Dv_i(\vn)$ for~$i=1, 2$. Besides, assume that
\[
\Dvtilde_3(\vntilde) = \Dv_3(\vntilde) = \Dv_3(\vn).
\]
Finally, for~$\alpha = 1, \dots, p-1$, let~$\{\qalpha\}_\alpha$ be the basis of~$\mathbb{P}_{p-1}(\E)/\mathbb{R}$ used in~\eqref{divergence:moment}.
We require
\begin{equation} \label{dof:bij}
\Dvtilde_4(\vntilde)_\alpha =  - \Dv_4(\vn)_\alpha + \int_{\partial \E} \vn \cdot \nE \qalpha.
\end{equation}
This implies that~$\Dv_4(\vntilde) = \Dv_4(\vn)$. Indeed, $\{\nabla \qalpha\}_{\alpha=1}^{p-2}$ is a basis for~$\Gcalpmt$ and 
\[
\begin{split}
\Dv_4(\vntilde)_\alpha 	& \overset{\eqref{divergence:moment}}{=} \int_\E\div(\vntilde) \qalpha \overset{\text{(IBP)\footnotemark}}{=}  - \int_{\E} \vntilde \cdot \nabla \qalpha  + \int_{\partial \E} \vntilde \cdot \nE \qalpha  \\
								& \overset{\eqref{divergence:momenttilde}}{=} -\Dvtilde_4(\vntilde)_\alpha + \int_{\partial \E} \vn\cdot \nE \qalpha \overset{\eqref{dof:bij}}{=}    \Dv_4(\vn)_\alpha .\\
\end{split}
\]
Using that~$\dim(\VnE) = \dim(\VnEtilde)$, we get that~$\TE$ is a bijection. \footnotetext{Here and in what follows (IBP) means 'integration by parts'.}
\end{proof}

As an immediate consequence, we have the following result.
\begin{cor}
The degrees of freedom~$\Dv_i$, $i=1, 2,3,4$ are unisolvent on~$\VnEtilde$.
\end{cor}

The two next lemmata are instrumental in order to prove Proposition~\ref{proposition:projT} below.
\begin{lem} \label{lemma:eq-polynomials}
Let~$\TE$ be the bijection introduced in Lemma~\ref{lemma:bijection}. Then, the following identity is valid:
\[
\int_{\E} \vn \cdot \qboldpmt = \int_{\E} (\TE\vn) \cdot \qboldpmt  \quad \quad \forall \vn\in\VnE , \, \forall \qboldpmt\in \PpmtEvec.
\]
\end{lem}
\begin{proof}
 For any~$\qboldpmt \in \PpmtEvec$, there exist unique~$\qpmo \in \mathbb{P}_{\p-1}(\E)/\mathbb{R}$ and~$\qboldtildepmt \in \Hcalpmt(\E)$ such that
\begin{equation} \label{the:splitting}
\qboldpmt =  \nabla \qpmo + \qboldtildepmt,
\end{equation}
see, e.g., \cite[Proposition~2.1]{bricksVEM}. Using Lemma~\ref{lemma:bijection}, we have~$\Dv_i(\vn) = \Dv_i(\TE\vn)$, $i=1,2,3,4$. Therefore, we deduce
\[
\begin{split}
\int_{\E} \vn \cdot \qboldpmt 	& \overset{\eqref{the:splitting}}{=}  \int_{\E} \vn \cdot \nabla \qpmo    + \int_{\E} \vn \cdot \qboldtildepmt   \\
										& \overset{\text{(IBP)}}{=} - \int_{\E} \div(\vn) \qpmo   + \int_{\partial\E} \vn\cdot\nE \qpmo   + \int_{\E} \vn \cdot \qboldtildepmt   \\
										& \overset{\eqref{equal:dofs}}{=} - \int_{\E} \div(\vntilde) \qpmo   + \int_{\partial\E} \vntilde\cdot\nE  \qpmo + \int_{\E} \vntilde \cdot \qboldtildepmt  \overset{\eqref{the:splitting}}{=} \int_{\E} \vntilde \cdot \qboldpmt.
\end{split}
\]
\end{proof}

\begin{lem}  \label{lem:projT}
Let~$\TE$ be the bijection introduced in Lemma~\ref{lemma:bijection}. Then, we have
\begin{equation}  \label{eq:projT}
\PinablapE (\TE\vn) = \PinablapE \vn , \qquad \PizpmtE (\TE\vn) = \PizpmtE \vn \quad \quad \forall \vn \in \VnE.
\end{equation}
\end{lem}
\begin{proof}
Let~$\vn\in \VnE$ and denote~$\vntilde = \TE\vn \in \Vntilde$. An integration by parts yields
\[
\aE (\qboldp, \PinablapE \vntilde) = - \int_{\E} \Delta \qboldp \cdot \vntilde + \int_{\partial \E} (\nabla \qboldp \nE) \cdot \vntilde \quad \quad \forall \qboldp \in \PpEvec.
\]
Since~$\vntilde{}_{|\partial \E}= \vn{}_{|\partial \E}$ and using Lemma~\ref{lemma:eq-polynomials}, we deduce
\[
\aE (\qboldp, \PinablapE \vntilde) =  - \int_{\E} \Delta \qboldp \cdot \vn + \int_{\partial \E} ( \nabla \qboldp \nE ) \cdot \vn = a^K(\qboldp, \PinablapE\vn).
\]
The second identity in~\eqref{eq:projT} is a direct consequence of Lemma~\ref{lemma:eq-polynomials}.
\end{proof}

Define the global bijection
\begin{equation} \label{eq:T}
\TF:\Vn\to\Vntilde 
\end{equation}
as~$\left(\TF\vn\right)_{|\E} = \TE(\vn{}_{|\E})$ for all~$\vn\in \Vn$ and~$\E \in \taun$.

The following result is a direct consequence of Lemmata~\ref{lemma:eq-polynomials} and~\ref{lem:projT}. 
\begin{prop} \label{proposition:projT}
For all~$\vn \in \Vn$, we have
\begin{equation} \label{eq:bT}
\b(\TF\vn, \qn) = \b(\vn, \qn),\qquad \forall \qn \in \Qn,
\end{equation}
and
\begin{equation} \label{eq:PiT}
\Pinablap (\TF\vn)  = \Pinablap \vn, \qquad \Pizpmt(\TF\vn) = \Pizpmt\vn.
\end{equation}
\end{prop}
\begin{proof}
The identities in~\eqref{eq:PiT} follow from Lemma~\ref{lem:projT} and the definitions of~$\Pinablap$ and~$\Pizpmt$ directly.
In order to show~\eqref{eq:bT}, remark that, due to Lemma~\ref{lemma:bijection},
\[
\int_\E \div(\TF\vn ) \qn = \int_\E \div(\vn) \qn \qquad \forall \E\in\taun, \quad \forall \qn \in \mathbb{P}_{p-1}(\E)/\mathbb{R}.
\]
Then, \eqref{eq:bT} follows from summing up the contributions of each integral in~$\E$.
\end{proof}

In words, Proposition~\ref{proposition:projT} states that, given two functions in the virtual element spaces~$\Vn$ and~$\Vntilde$ sharing the same value of the degrees of freedom,
their~$\Pinablap$ and~$\Pizpmt$ projections, as well as their evaluations through~$\b(\cdot, \qn)$ for all~$\qn \in \Qn$, are the same.

%%%%%%%%%%%%%%%%%%%%%%%%%%%%%%%%%%%%%%
\section{A priori estimates} \label{section:a-priori}
%%%%%%%%%%%%%%%%%%%%%%%%%%%%%%%%%%%%%%
In this section, we prove the well-posedness and provide an abstract error analysis of method~\eqref{Stokes:VEM}.
To this aim, we first prove that the bilinear form~$\b(\cdot,\cdot)$ satisfies a discrete inf-sup condition independently of the degree of accuracy of the method; see Section~\ref{subsection:discrete-infsup}.
Secondly, in Section~\ref{subsection:stabilization}, we analyse the discrete bilinear form~$\an(\cdot, \cdot)$ and show that, under suitable assumptions on the stabilization terms, it is coercive and continuous.
Notably, the coercivity and continuity constants are determined using  Poisson-like spaces and are explicit in terms of the degree of accuracy $p$ of the method.
The abstract error analysis on the velocities and pressures is provided in Sections~\ref{sec:apriori-u} and~\ref{subsection:apriori-p}, respectively.
The bounds herein proven are instrumental in deducing the rate of convergence of the error of the method, which is the topic of Section~\ref{section:convergence} below.

%%%%%%%%%%%%%%%
\subsection{The discrete inf-sup condition} \label{subsection:discrete-infsup}
%%%%%%%%%%%%%%%
The discrete inf-sup stability of method~\eqref{Stokes:VEM} has been shown in~\cite{BLV_StokesVEMdivergencefree} already.
Here, we recall its proof, and show that the discrete inf-sup constant is independent of the degree of accuracy~$\p$.

We start by recalling a classical result on the inf-sup constant for star-shaped domains.
\begin{lem} \label{lemma:Costabel-Dauge}
Let~$D \subset \mathbb R^2$ be a domain contained in a ball of radius~$R$ and star-shaped with respect to a concentric ball of radius~$\rho$.
Denote the inf-sup constant of~$\b^D(\cdot, \cdot)$ by~$\beta(D)$. Then, the following lower bound is valid:
\[
\beta(D) \ge \frac{\rho}{2R}.
\]
\end{lem}
\begin{proof}
See~\cite[Theorem 2.3]{costabel2015inequalities}.
\end{proof}

\begin{lem} \label{lemma:discrete-infsup}
There exists a constant~$\beta_n$, independent of the element sizes and of the degree of accuracy~$\p$, such that
\begin{equation} \label{eq:discrete-infsup}
\inf_{q_n\in Q_n} \sup_{\vn\in \Vn} \frac{b(\vn, q_n)}{|\vn|_{1,\Omega}\|q_n\|_{0, \Omega}} \geq \beta_n.
\end{equation}
\end{lem}
\begin{proof}
As is customary, we use Fortin's trick, i.e., we show the existence of an operator~$\Pi_n :\Vbf \to \Vn$ and a positive constant~$C$ independent of~$\p$ such that
\[
\begin{cases}
b(\Pi_n \vbf, q_n)  = b(\vbf, q_n) \text{ for all }q_n\in Q_n\\
\| \Pi_n \vbf \|_{1,\Omega}\leq C \| \vbf\|_{1, \Omega}.
\end{cases}
\]
This implies the validity of the inf-sup stability of the spaces~$\Vn$ and~$Q_n$; see, e.g., \cite{BrezziFortin}.
We devote the remainder of the proof to showing the existence of such operator~$\Pi_n$ and constant~$C$.
\medskip

Let~$\Wn$ be a low-order ($\p=2$) virtual element space for the velocity. By~\cite[Proposition 4.2]{BLV_StokesVEMdivergencefree}, there exists $\vnbar\in \Wn$ such that
\begin{equation} \label{eq:vnbar}
\begin{cases}
\b(\vnbar, \qnbar)  = \b(\vbf, \qnbar) \text{ for all }\qnbar \in \mathbb{P}_0(\mesh)\\
\| \vnbar \|_{1,\Omega}\leq C \| \vbf\|_{1, \Omega}.
\end{cases}
\end{equation}
In each element~$\E$, we introduce a bubble function~$\wnE\in \VnE$ such that
\begin{itemize}
\item $\wnE{}|_{\partial K}= 0$;
\item $\displaystyle\int_K \wnE \qalphaperp = 0$ for all~$\qalphaperp \in \Hcalp(\E)$;
\item $\displaystyle\int_K \div(\wnE) \qpmo = \displaystyle\int_K \div(\vbf-\vnbar) \qpmo$ for all~$\qpmo \in \mathbb{P}_{\p-1}(\E)/\mathbb{R}$.
\end{itemize}
In other words, we construct~$\wnE$ such that, in each element~$\E \in \taun$, $\Dv_i(\wnE)= 0$, $i=1,2,3$, and~$\Dv_4(\wnE) = \Dv_4(\vbf-\vnbar)$.
Besides, by the definition of the space~$\VnE$, there exist~$s\in   L^2(\E)$ such that
\[
\begin{cases}
-\Delta \wnE - \nabla s = 0 &\text{in } \E\\
\div \wnE = \PizpmoE  \div(\vbf - \vnbar) &\text{in } \E.\\
\end{cases}
\]
By the standard well-posedness of the above Stokes problem, we claim that 
\begin{equation} \label{eq:wnE-stab}
 | \wnE |_{1, \E} \leq \frac{1}{\beta(K)}\| \PizpmoE  \div(\vbf - \vnbar) \|_{0,K} \leq  \frac{1}{\beta(K)}| \vbf - \vnbar|_{1,K}.
\end{equation}
In order to show~\eqref{eq:wnE-stab}, first observe that
\[
|\wnE|^2_{1, \E} = a^K(\wnE, \wnE) = -b^K(\wnE, s) \leq\| \PizpmoE  \div(\vbf - \vnbar) \|_{0,K} \|s\|_{0,K}.
\]
Next, denote the inf-sup constant of the continuous Stokes problem in~$\E$ with homogeneous Dirichlet boundary conditions by~$\beta(\E)$. This gives
\[
\|s\|_{0,K}\leq \frac{1}{\beta(K)} \sup_{\vbf\in H^1_0(K)^2}\frac{b^K(\vbf, s)}{|\vbf|_{1,K}} =  \frac{1}{\beta(K)} \sup_{\vbf\in H^1_0(K)^2}\frac{a^K(\wnE, \vbf)}{|\vbf|_{1,K}} \leq \frac{1}{\beta(K)} | \wnE|_{1,K},
\]
whence~\eqref{eq:wnE-stab} follows.
\medskip

Next, consider~$\wn \in \Vn$ defined as~$\wnE$ in each element~$\E \in \taun$ and define~$\Pi_n \vbf = \wn + \vnbar$.
By construction, it follows that
\[
\b(\Pi_n \vbf, \qn)  = \b(\vbf, \qn) \quad \quad \forall \qn\in \Qn.
\]
From~\eqref{eq:vnbar} and~\eqref{eq:wnE-stab}, we deduce that~$\Pi_n$ is~$H^1(\Omega)$-stable, with stability constant independent of the degree of accuracy~$\p$.
\end{proof}

%%%%%%%%%%%%%%%
\subsection{Stabilization, coercivity, and continuity: well-posedness of the VEM} \label{subsection:stabilization}
%%%%%%%%%%%%%%%
In this section, we analyse the properties of the discrete bilinear form~$\an(\cdot,\cdot)$.
Notably, we show that suitable choices of the stabilization forms yield to a coercive and continuous bilinear form.
Furthermore, the coercivity and continuity constant are explicit in terms of the degree of accuracy of the method~$\p$.
The main ingredient is given by the properties of the bijection~$\TE$; see Lemma~\ref{lemma:bijection}.

In order to investigate the stability of the method, we require an additional property on the stabilization bilinear forms:
For all~$\un, \vn\in\VnE$ and~$\untilde, \vntilde\in \VnEtilde$ such that~$ \Dv_i(\vn) =  \Dv_i(\vntilde)$, $i=1,2,3,4$, i.e., $\vntilde = \TF\vn$ with~$\TF$ defined in Lemma~\ref{lemma:bijection},
\begin{equation} \label{eq:S-hypothesis}
\SE(\un , \vn) = S^{\E}(\untilde, \vntilde)
\end{equation}
Furthermore, we assume that, for all~$\p\in \mathbb{N}$ and~$\E\in \taun$,
there exist two positive constant~$\alphacoerhat < \alphaconthat$, such that
\begin{equation} \label{eq:stab-constants}
\SE (\vntilde, \vntilde) \geq \alphacoerhat | \vntilde |^2_{1, \E},\quad S^{\E}(\untilde, \vntilde) \leq \alphaconthat |\untilde|_{1,\E}|\vntilde|_{1,\E} \quad \forall \untilde, \vntilde\in \VnEtilde\cap \ker(\PinablapE). 
\end{equation}
Set
\[
\alphacoer = \min(1, \alphacoerhat), \qquad \alphacont = \max(1, \alphaconthat).
\]
Following, e.g., \cite{VEMvolley}, we can prove that~$\alphacoer$ and~$\alphacont$ are the coercivity and continuity constants for the discrete bilinear form~$\an(\cdot, \cdot)$.
The actual dependence on~$\p$ of the two constants hinges upon the definition of the stabilizing bilinear forms~$\SE(\cdot, \cdot)$ in~\eqref{local:discrete:BF};
see Remark~\ref{remark:explicit-STAB} below for an explicit choice of the stabilization together with the explicit dependence in terms of the degree of accuracy.

As in~\cite{VEMvolley}, the properties of the discrete bilinear form~$\an(\cdot, \cdot)$ entail that the method is stable and $\p$-polynomially consistent.
We have the following well-posedness result.
\begin{thm} \label{theorem:WP-VEM}
Method~\eqref{Stokes:VEM} is well-posed.
\end{thm}
\begin{proof}
The assertion follows from the continuity of the bilinear form~$\an$ and~$\bn$, the coercivity of~$\an$, the discrete inf-sup condition~\eqref{eq:discrete-infsup}, and standard argument as in~\cite{BrezziFortin}.
\end{proof}

\begin{remark} \label{remark:explicit-STAB}
An example of an explicit stabilization~$S^\E$ such that~\eqref{eq:S-hypothesis} and~\eqref{eq:stab-constants} are valid is as follows:
\begin{equation} \label{eq:SEexplicit}
\SE (\un, \vn) = \frac{p}{h_{\E}}(\un, \vn)_{0, \partial \E} + \frac{p^2}{h_{\E}^2}\left( \PizpmtE\un, \PizpmtE\vn \right)_{0, K} \quad \quad \forall \un, \vn \in \VnE.
\end{equation}
All the terms on the right-hand side of~\eqref{eq:SEexplicit} are computable via the degrees of freedom~$\Dv_i$, $i=1, \dots, 4$ explicitly.
Furthermore, \eqref{eq:S-hypothesis} is valid thanks to Lemmata~\ref{lemma:bijection} and~\ref{lem:projT}.
On the other hand, the bounds in~\eqref{eq:stab-constants} can be proven as in~\cite[Theorem 2]{hpVEMcorner}, with explicit stability constants
\[
\alphacoerhat \geq p^{-5} , \qquad \alphaconthat \leq
\begin{cases}
1&\text{if~$\E$ is convex}\\
\p^{2\left( 1-\frac{\pi}{\omega_{\E}} +\epsilon\right)}&\text{otherwise},
\end{cases}
\]
for all~$\epsilon>0$ and where~$\omega_{\E}$ denotes the largest angle of~$\E$.

In all fairness, the practical dependence of the stabilization constants in terms of~$\p$ results to be much milder numerically; see~\cite[Section~4.6]{hpVEMbasic} and~\cite[Section~4.1]{hpVEMcorner}.
\end{remark}

\subsubsection*{Why did we assume~\eqref{eq:S-hypothesis}?}
The reason we have introduced the auxiliary Poisson-like virtual element space~$\Vntilde$ in~\eqref{Poisson-like:VEM} and analysed its relation with the Stokes-like virtual element space~$\Vn$ in~\eqref{VEspace:velocity:global}
is that we can exploit previous stability bounds that are explicit in terms of the degree of accuracy~$\p$; see~\cite[Section~4]{hpVEMcorner}.

Notably, the nonstandard assumption~\eqref{eq:S-hypothesis}, together with~\eqref{eq:bT} and~\eqref{eq:PiT}, allows us to analyse method~\eqref{Stokes:VEM}
mapping Stokes-like virtual element functions into Poisson-like ones.

%%%%%%%%%%%%%%%
\subsection{A priori estimate on the velocity} \label{sec:apriori-u}
%%%%%%%%%%%%%%%
In this section, we prove some upper bounds, which will be instrumental in the analysis of the convergence for the error on the velocity.

Introduce the weakly divergence-free subspace of~$\Vntilde$
\[
\Zntilde := \left\{ \vntilde\in \Vntilde: b(\vntilde, \qn) = 0 \text{ for all } \qn \in \Qn \right\}.
\]
For future use, we also introduce the weakly divergence-free subspace of~$\Vn$
\begin{equation} \label{nucleus-Vn}
\Zn := \left\{ \vn \in \Vn : \b(\vn, \qn) = 0 \text{ for all } \qn \in \Qn \right\}.
\end{equation}
Moreover, let $\Fcaln$ denote the smallest constant such that
\[
|((\Id - \Pizpmt) \fbf, \vntilde)_{0, \Omega}| \leq \Fcaln |\vntilde|_{1, \taun} , \qquad \forall \vntilde \in \Zntilde.
\]
The first result is an upper bound on the error between the solution to the continuous problem and the discrete solution mapped through the bijection in~\eqref{eq:T}.
\begin{lem} \label{lemma:apriori-untilde}
Let~$\ubf$ be the solution to~\eqref{Stokes:weak}, $\un\in \Vn$ be the virtual element solution to~\eqref{Stokes:VEM}, and~$\TF$ be the bijection defined in~\eqref{eq:T}.
Then, the following bound is valid:
\begin{equation} \label{eq:apriori-untilde}
| \ubf - \TF\un |_{1, \taun} \leq \frac{1}{\alphacoer}\left( \Fcaln + (\alphacont+1)\left( \inf_{\zntilde\in \Zntilde}|\ubf-\zntilde|_{1, \Omega} + \inf_{\upi\in [\mathbb{P}_p(\taun)]^2}| \ubf - \upi|_{1, \taun}  \right)\right).
\end{equation}
\end{lem}
\begin{proof}
Introduce~$\untilde = \TF \un$. Since~$\b(\un, \qn)=0$ for all~$\qn \in \Qn$, use~\eqref{eq:bT} to get that~$\untilde \in \Zntilde$.
Moreover, by~\eqref{eq:PiT} and~\eqref{eq:S-hypothesis}, $\untilde$ is the solution to the reduced problem 
\[
\begin{cases}
\text{find }\untilde \in \Zntilde\text{ such that}\\
\an (\untilde, \vntilde) = (\Pizpmt\fbf, \vntilde) &\forall \vntilde\in \Zntilde.
\end{cases}
\]
In fact, $\un$ solves the Stokes-like counterpart
\[
\begin{cases}
\text{find }\un \in \Zn \text{ such that}\\
\an (\un, \vn) = (\Pizpmt\fbf, \vn) &\forall \vn \in \Zn.
\end{cases}
\]
The analysis proceeds with classical tools for a priori estimates for virtual element methods; see, e.g., \cite{VEMvolley}.
For any~$\zntilde\in \Zntilde$, the triangle inequality yields
\begin{equation}   \label{eq:aprioriz-2}
| \ubf - \untilde |_{1, \taun} \leq  | \ubf - \zntilde |_{1, \taun} +  | \zntilde - \untilde |_{1, \taun} .
\end{equation}
Denoting~$\deltan = \zntilde-\untilde\in \Zntilde$, we compute, for all~$\upi\in [\mathbb{P}_{p}(\taun)]^2$,
\begin{align*}
\alphacoer |\deltan|	& ^2_{1, \taun}  \leq \sumE a_n^{\E}(\deltan, \deltan)  \\
							& =\sumE \left( a^\E(\ubf, \deltan)- a_n^\E(\untilde, \deltan) \right) + \sumE a_n^{\E}(\zntilde - \upi, \deltan) + \sumE a^{\E}(\upi -\ubf, \deltan) \\
							& \leq  ((\Id- \Pizpmt)\fbf, \deltan)_{0, \Omega}  +\alphacont \sumE|\zntilde - \upi|_{1, \E}| \deltan|_{1,\E} +\sumE |\upi -\ubf|_{1, \E} |\deltan|_{1,\E} \\
							& \leq  \left(\Fcaln  + \alphacont|\zntilde - \upi|_{1, \taun} +|\upi -\ubf|_{1, \taun}   \right)| \deltan|_{1,\taun},
\end{align*}
where the last inequality follows from the definition of~$\Fcaln$ and from the Cauchy-Schwarz inequality.

Dividing both sides by~$| \deltan|_{1,\taun}$ gives
\begin{equation} \label{eq:aprioriz-3}
|\zntilde - \untilde |_{1, \taun} \leq \frac{1}{\alphacoer}\left( \Fcaln + \alphacont |\zntilde - \ubf|_{1, \taun} + (\alphacont +1)| \upi - \ubf|_{1, \taun}\right).
\end{equation}
The assertion follows combining~\eqref{eq:aprioriz-2} and~\eqref{eq:aprioriz-3}.
\end{proof}

The next result is an upper bound on the error between the solution to the continuous problem and the $H^1$ projection of the discrete Stokes-like solution.
\begin{lem} \label{lemma:apriori-u}
Let~$\ubf$ and~$\un\in \Vn$ be the solutions to~\eqref{Stokes:weak} and~\eqref{Stokes:VEM}, respectively.
Then, we have
\begin{equation}  \label{eq:apriori-u}
\begin{split}
& | \ubf - \PinablapE\un |_{1, \taun}  \\
& \leq \frac{1}{\alphacoer}\left( \Fcaln + (\alphacont+1) \inf_{\zntilde\in \Zntilde}|\ubf-\zntilde|_{1, \Omega} + (\alphacont +2)\inf_{\upi\in [\mathbb{P}_p(\taun)]^2}| \ubf - \upi|_{1, \taun}\right). \\
\end{split}
\end{equation}
\end{lem}
\begin{proof}
Let~$\untilde = \TF \un \in \Vntilde$. Use Proposition~\ref{proposition:projT} to get
\[
\PinablapE \un {}_{|\E} = \PinablapE \untilde {}_{|\E}  \quad\text{ for all }\E\in \taun.
\]
The triangle inequality and the stability of the~$H^1$ projector give
\begin{equation} \label{eq:aprioriz-1}
\begin{aligned}
| \ubf - \PinablapE\un |_{1, \taun} 	&  \leq   | \ubf - \PinablapE\ubf |_{1, \taun} +  \left(\sum_{\E\in \taun}| \PinablapE(\ubf - \untilde) |^2_{1, \E} \right)^{1/2}  \\     																
											&  \leq  | \ubf - \PinablapE\ubf |_{1, \taun}   +    | \ubf - \untilde  |_{1, \taun} .
\end{aligned}
\end{equation}
For all~$\upi \in [\mathbb{P}_p(\taun)]^2$, we have
\begin{equation} \label{eq:aprioriz-4}
| \ubf - \Pinablap\ubf|_{1, \taun} \leq |\ubf - \upi|_{1, \taun}.
\end{equation}
Combining~\eqref{eq:apriori-untilde}, \eqref{eq:aprioriz-1}, and~\eqref{eq:aprioriz-4}, the assertion follows.
\end{proof}
 
Next, we show an upper bound on the best error on the Poisson-like weakly divergence free subspace~$\Zntilde$ in terms of a best error in terms of functions in the Poisson-like virtual element space~$\Vntilde$.
\begin{lem} \label{lemma:bound-nucleus}
Let~$\ubf\in [H^1_0(\Omega)]^2$ be such that
\begin{equation} \label{nucleus}
\b(\ubf, \q) = 0 \quad \quad  \forall \q\in L^2_0(\Omega).
\end{equation}
Then, the following upper bound is valid:
\[
\inf_{\zntilde\in \Zntilde} | \ubf - \zntilde|_{1, \Omega}\leq \left( 1+ \left(\frac{1+\alphacont}{\alphacoer}\right)^{1/2}\right)\inf_{\vntilde\in \Vntilde}| \ubf - \vntilde|_{1, \Omega}  .
\]
\end{lem}
\begin{proof}
We begin by proving a discrete ``switched inf-sup'' condition.
Introduce~$\Zn^C$ the complementary space of~$\Zn$ defined in~\eqref{nucleus-Vn} in~$\Vn$.
In Lemma~\ref{lemma:discrete-infsup}, we proved the existence of a surjective operator~$\dalethn : \Qn \rightarrow \Vn$ such that
\begin{equation} \label{supinf-1}
(\dalethn \qn, \vn)_{1,\Omega} = \b(\vn,\qn) \quad \quad \forall \vn \in \Zn^C, \, \forall \qn \in \Qn.
\end{equation}
In particular, the discrete inf-sup condition~\eqref{eq:discrete-infsup} can be written as
\begin{equation} \label{supinf-2}
\beta_n \Vert \qn \Vert _{0,\Omega} \le \vert \dalethn \qn \vert _{1,\Omega} \quad \quad \forall \qn \in \Qn.
\end{equation}
Thence, for all~$\vn \in \Zn^C$, thanks to the surjectivity of~$\dalethn$, we can write
\begin{equation} \label{supinf-3}
\begin{split}
\beta_n \vert \vn \vert_{1,\Omega} 	& = \beta_n \sup_{\vntilde \in \Zn^C} \frac{(\vntilde,\vn)_{1,\Omega}}{\vert \vntilde \vert_{1,\Omega}} = \beta_n \sup_{\qn \in \Qn} \frac{(\dalethn \qn, \vn)_{1,\Omega}}{\vert \dalethn \qn \vert_{1,\Omega}}\\
												& \overset{\eqref{supinf-2}}{\le} \sup_{\qn \in \Qn} \frac{(\dalethn \qn, \vn)_{1,\Omega}}{\Vert \qn \Vert_{0,\Omega}} \overset{\eqref{supinf-1}}{=} \sup_{\qn \in \Qn} \frac{\b(\vn,\qn)}{\Vert \qn \Vert_{0,\Omega}}.
\end{split}
\end{equation}
For each~$\vntilde\in \Vntilde$, define~$\wn\in \Zn^C$ as the solution of
\begin{equation} \label{eq:wn}
\begin{cases}
\text{find }\wn\in \Zn^C\text{ such that }\\
 \b(\wn, \qn) = \b(\vntilde, \qn) & \forall \qn \in \Qn.
\end{cases}
\end{equation}
This problem has a unique solution due to the continuity and the discrete ``switched inf-sup'' stability in~\eqref{supinf-3} of the bilinear form~$b(\cdot, \cdot)$; see, e.g., \cite{BrezziFortin}.
Furthermore, the following a priori estimate is valid:
\begin{equation} \label{estimate:wn}
|\wn|_{1,\Omega} \overset{\eqref{supinf-3}}{\leq} \frac{1}{\beta_n} \sup_{q_n\in Q_n}\frac{b(\wn, q_n)}{\|q_n\|_{0,\Omega}}  
\overset{\eqref{nucleus}, \eqref{eq:wn}}{=} \frac{1}{\beta_n} \sup_{q_n\in Q_n}\frac{b(\vntilde - \ubf, q_n)}{\|q_n\|_{0,\Omega}}  \leq \frac{1}{\beta_n} |\vntilde - \ubf|_{1,\Omega}.
\end{equation}
Next, define
\begin{equation} \label{ztilde}
\zntilde = \vntilde - \TF\wn,
\end{equation}
where~$\TF$ is the bijection in~\eqref{eq:T}.
Thanks to~\eqref{eq:bT}, we get
\[
\b(\zntilde, \qn)   = \b(\vntilde - \TF\wn, \qn) = \b(\vntilde-\wn, \qn) = 0 \qquad \forall \qn \in \Qn.
 \]
We deduce that~$\zntilde\in \Zntilde$.
Then, we have
\begin{equation} \label{estimate:Twn}
\alphacoer|\TF\wn|^2_{1, \Omega} \overset{\eqref{eq:stab-constants}}{\leq} a_n(\TF\wn, \TF\wn) \overset{\eqref{eq:PiT}, \eqref{eq:S-hypothesis}}{=} \an (\wn, \wn) \overset{\eqref{eq:stab-constants}}{\leq} (1+\alphacont)|\wn|^2_{1, \Omega}.
\end{equation}
This yields
\[
\begin{split}
|\zntilde - \ubf|_{1, \Omega}	& \overset{\eqref{ztilde}}{\leq} |\vntilde- \ubf|_{1, \Omega} + |\TF\wn|_{1, \Omega} \overset{\eqref{estimate:Twn}}{\leq} |\vntilde- \ubf|_{1,\Omega}+\left(\frac{1+\alphacont}{\alphacoer}\right)^{1/2}|\wn|_{1,\Omega} \\
									& \overset{\eqref{estimate:wn}}{\leq} \left( 1+\frac{1}{\beta_n}\left(\frac{1+\alphacont}{\alphacoer}\right)^{1/2} \right)|\vntilde - \ubf|_{1,\Omega},\\
\end{split}
\]
whence the assertion follows.
\end{proof}
 
\begin{remark}
The last part of the proof of Lemma~\ref{lemma:bound-nucleus} also gives
\[
 \inf_{\qn \in \Qn} \sup_{\vntilde \in \Vntilde} \frac{\b(\vntilde,
   \qn)}{|\vn|_{1,\Omega}\| \qn\|_{0, \Omega}}\geq \beta_n\sqrt{\alphacoer}/\sqrt{1+\alphacont} .
\]
\end{remark}
 
%%%%%%%%%%%%%%%
\subsection{A priori estimate on pressure} \label{subsection:apriori-p}
%%%%%%%%%%%%%%%
In this section, we prove some upper bounds which will be instrumental in the
analysis of the convergence of the error on the pressure obtained by
  the VEM.
\begin{lem}
\label{lemma:apriori-p}
Let~$(\ubf, s)\in [H^1_0(\Omega)]^2 \times L^2_0(\Omega)$  and~$(\un, \sn)\in  \Vn\times \Qn$ be the solutions to~\eqref{Stokes:weak} and~\eqref{Stokes:VEM}, respectively. 
Recall that the bijection~$\TF$ is defined in~\eqref{eq:T}. Then, the following bound is valid:
\begin{equation}  \label{eq:apriori-p}
\begin{split}
\|s- \sn\|_{0,\Omega}\leq \frac{1}{\beta_n}\bigg( 	& \Fcaln + (1+\beta_n) \inf_{\qn \in \Qn}\|s-\qn\|_{0, \Omega}   \\
																& + (1+\alphacont)| \ubf - \TF\un|_{1, \Omega}+(2+\alphacont)\inf_{\upi\in [\mathbb{P}_p(\taun)]^2}|\ubf-\upi|_{1, \taun}\bigg).
\end{split}
\end{equation}
\end{lem}
\begin{proof}
For all~$\qn\in \Qn$, the triangle inequality yields 
\[
\| s - \sn \|_{0, \Omega}\leq \| s - \qn \|_{0, \Omega} +\| \sn - \qn \|_{0, \Omega}.
\]
By the discrete inf-sup condition~\eqref{eq:discrete-infsup}, there exists~$\vn\in \Vn$ such that
\[
\beta_n\| \sn - \qn \|_{0, \Omega}\leq \frac{ \b(\vn, \sn - \qn)}{|\vn|_{1, \Omega}}.
\]
We have
\[
\b(\vn, \sn - \qn) = \b(\vn, s - \qn) + \b(\vn, \sn- s)
\]
and
\[
 |b(\vn, s- \qn)|\leq |\vn|_{1, \Omega} \| s - \qn\|_{0,\Omega}.
\]
For $\untilde = \TF\un$, we deduce
\[
\begin{split}
 |\b(\vn, \sn-s)|	& \leq |\a(\ubf, \vn)  - \an(\un, \vn)| + |(\fbf - \Pizpmt\fbf,\vn)|  \\
					& = |\a(\ubf, \vn)  - \an(\untilde, \vn)| + |(\fbf - \Pizpmt\fbf,\vn)| \\
 				   	& \leq |\sumE a^{\E}(\ubf-\upi, \vn)| + |\sumE \anE (\upi-\untilde, \vn)| + \Fcaln |\vn|_{1, \Omega} \\
					& \leq \left((2+\alphacont)|\ubf-\upi|_{1, \taun}   + (1+\alphacont) |\ubf-\untilde|_{1, \Omega} + \Fcaln\right)|\vn|_{1,\Omega},
\end{split}
\]
whence the assertion follows.
\end{proof}
Define
\begin{equation} \label{gammap}
\gamma(\p) = \dfrac{\alphacont+1}{\alphacoer}.
\end{equation}
Combining Lemmata~\ref{lemma:apriori-untilde}, \ref{lemma:apriori-u}, and~\ref{lemma:apriori-p}, we obtain the following result.
\begin{thm} \label{theorem:abstract}
Let~$(\ubf, s)\in [H^1_0(\Omega)]^2 \times L^2_0(\Omega)$  and~$(\un, \sn)\in  \Vn\times \Qn$ be the solutions to~\eqref{Stokes:weak} and~\eqref{Stokes:VEM}, respectively.
Recall that~$\gamma(\p)$ is defined in~\eqref{gammap}.
Then, there exists a constant~$C>0$ independent of the discretization parameters such that
\begin{equation} \label{eq:apriori-up}
\begin{split}
|\ubf - \Pinablap \un|_{1, \taun} + \beta_n \|s -\sn\|_{0, \Omega}\leq C 	& \gamma(\p) \bigg(  \Fcaln +  \sqrt{\gamma(\p)}\inf_{\vntilde\in \Vntilde}|\ubf-\vntilde|_{1, \Omega}  \\
																							& + \inf_{\upi\in [\mathbb{P}_{\p}(\taun)]^2}|\ubf-\upi|_{1,\taun} + \inf_{\qn \in \Qn}\|s-\qn\|_{0, \Omega}  \bigg).
\end{split}
\end{equation}
\end{thm}

%%%%%%%%%%%%%%%%%%%%%%%%%%%%%%%%%%%%
\section{The convergence rate of the $\p$- and $\h\p$-versions} \label{section:convergence}
%%%%%%%%%%%%%%%%%%%%%%%%%%%%%%%%%%%%
In Section~\ref{section:a-priori}, we have established an abstract error analysis for method~\eqref{Stokes:VEM}.
Notably, we have proven that the error on the velocity and the pressure can be estimated from above in terms of best polynomial approximation
and best interpolation results in virtual element spaces.
With this at hand, in this section, we state the convergence of the $\p$- and
$\h\p$-versions of method~\eqref{Stokes:VEM} for analytic, weighted
  analytic, and finite Sobolev regularity solutions;
see Sections~\ref{subsection:pVE-spaces} and~\ref{subsection:hpVE-spaces}, respectively.

%%%%%%%
\subsection{$\p$-VEM} \label{subsection:pVE-spaces}
%%%%%%%
Since all the necessary best approximation results have been proven in~\cite{hpVEMbasic}, we state the main convergence result only.
\begin{thm} \label{theorem:p-version}
Let~$k \in \mathbb R^+$ be such that $(\ubf, s)\in [H^1_0(\Omega) \cap H^{k+1}(\Omega)]^2 \times [H^k(\Omega) \cap L^2_0(\Omega)]$
and~$(\un, \sn)\in  \Vn\times \Qn$ are the solutions to~\eqref{Stokes:weak} and~\eqref{Stokes:VEM}, respectively.
Let the assumptions (\textbf{A0-$\p$}), (\textbf{A1}), and (\textbf{A2}) be valid.
Recall that~$\gamma(\p)$ is defined in~\eqref{gammap}.
Then, there exists a positive constant~$C$ independent of the discretization parameters such that
\begin{equation} \label{algebraic-convergence:p-version}
|\ubf - \Pinablap \un|_{1, \taun} + \beta_n \|s -\sn\|_{0, \Omega}\leq C \gamma(\p)^{\frac{3}{2}} \frac{\h^{\min(k,\p)}}{\p^k} \left( \Vert u \Vert_{k+1, \Omega} + \Vert s \Vert_{k,\Omega} \right).
\end{equation}
Furthermore, if $\ubf$ and~$s$ are the restrictions of suitable analytic functions over an extension of the domain\footnote{See~\cite[Section~5]{hpVEMbasic} for more details on this point.}~$\Omega$,
then there exist two positive constants~$C_1$ and~$C_2$ independent of the discretization parameters such that
\begin{equation}   \label{exponential-convergence:p-version}
|\ubf - \Pinablap \un|_{1, \taun} + \beta_n \|s -\sn\|_{0, \Omega}\leq C_1 \exp(-C_2 \, \p).
\end{equation}
\end{thm}
\begin{proof}
Starting from the abstract error analysis in Theorem~\ref{theorem:abstract}, it suffices to be able to show $\h$- and $\p$-upper bounds on the four terms appearing on the right-hand side of~\eqref{eq:apriori-up}.
We can show an upper bound on them using~\cite[Lemmata~4.2, 4.3, and~4.4]{hpVEMbasic} for the finite Sobolev regularity case,
and~\cite[Lemmata~5.2, 5.3, and~5.4]{hpVEMbasic} for the analytic regularity case.

The bound~\eqref{algebraic-convergence:p-version} follows in a straightforward manner,
whereas, in order to prove~\eqref{exponential-convergence:p-version}, we apply similar results as in~\cite[Theorem~5.2]{hpVEMbasic}.
\end{proof}

From Theorem~\ref{theorem:p-version}, we have that the $\p$-version of the
method converges
exponentially for analytic solutions and algebraically for solution with (sufficiently high) finite Sobolev regularity.
However, since solutions to the Stokes problem are in general singular, as detailed in Theorem~\ref{theorem:regularity}, we are also interested in analysing the convergence of the $\h\p$-version of the method.
Indeed, it is known that such approach allows for exponential convergence with respect to a suitable root of the total number of degrees of freedom for singular solutions as well.
We postpone the design of $\h\p$-virtual element spaces for the Stokes problem, as well as the convergence of the error, to Section~\ref{subsection:hpVE-spaces} below.

\begin{remark}
An additional reason why the $\h\p$-version is more suited than the~$\p$-version for the approximation of singular solutions to the Stokes problem
is that the algebraic rate of convergence in~\eqref{algebraic-convergence:p-version} contains the suboptimal term~$\gamma(\p)$ due to the stabilization of the method.
\end{remark}

\begin{remark}
In Theorem~\ref{theorem:p-version}, we proved upper bounds for errors of the form
\begin{equation} \label{computable:errors}
|\ubf - \Pinablap \un|_{1, \taun} + \beta_n \|s -\sn\|_{0, \Omega},
\end{equation}
which differ from those that are typically investigate in the VEM literature, i.e.,
\[
|\ubf - \un|_{1, \taun} + \beta_n \|s -\sn\|_{0, \Omega}.
\]
The reason for this is that we need to resort to Poisson-like spaces, when performing the theoretical analysis, and we know from Proposition~\ref{proposition:projT}
that functions in Poisson-like and Stokes-like virtual element spaces, sharing the same degrees of freedom, have the same~$\Pinablap$ projection.
In turn, we had to resort to Poisson-like virtual element spaces, because we are not able to construct a stabilization on Stokes-like virtual element spaces,
with bounds on the stabilization constants, which are explicit in terms of the degree of accuracy of the method.
On the positive side, the two errors that we bound are those that we actually compute in the numerical experiments presented in Section~\ref{section:numerical-experiments} below.
\end{remark}

%%%%%%%
\subsection{$\h\p$-VEM} \label{subsection:hpVE-spaces}
%%%%%%%
In the present section, we construct $\h\p$-virtual element spaces for the approximation of nonsmooth solutions to the Stokes problem~\eqref{Stokes:weak}.
The main idea of the construction hinges upon employing
\begin{itemize}
\item geometric refinement of the mesh towards the singular points;
\item $\p$-refinement in the elements where the solution is smooth.
\end{itemize}

For the sake of exposition, assume that the right-hand side~$\fbf$ in~\eqref{Stokes:weak} is smooth.
Thanks to Theorems~\ref{theorem:regularity} and~\ref{theorem:regularity2}, the solution~$(\ubf,s)$ to~\eqref{Stokes:weak} consists
of two functions that are smooth everywhere but at neighbourhoods of the vertices of the polygonal domain~$\Omega$.
There, the Sobolev regularity is known a priori and depends on the amplitude of the angle.

The first step in the construction of $\h\p$-virtual element spaces resides in introducing the layer of the mesh associated with the set of vertices~$\fC$.
We assume that the mesh~$\taun$ consists of~$n+1$ layers, where the first one is given by
\[
L_n^0 := \{ \E \in \taun \mid \text{there exists a unique } \fc \in \fC \text{ such that } \fc \in \EE \},
\]
and the others are defined recursively as
\[
L_n^j := \{ \E \in \taun \mid \E \not\in \cup _{\ell=0}^{j-1} L_n^\ell  ; \; \exists \Etilde \in L_n^{j-1} \text{ such that } \overline \E \cap \overline{\Etilde} \ne \emptyset  \} \quad \quad \forall j=1,\dots, n.
\]
Further, for each $K\in \taun$, we denote any of the closest corner of the domain to $K$, i.e., any of the $\fc\in\fC$ such that
$\dist(\fc, K)\leq \dist(\tilde{\fc}, K)$ for all $\tilde{\fc}\in \fC\setminus \{\fc\}$, by~$\fc_K$.
For the sake of simplicity, we assume the uniqueness of such a vertex.

With this at hand, we say that the sequence of meshes~$\{ \taun \}_{n\in \mathbb N}$ is geometrically refined towards~$\fc_K$
if there exists a grading parameter~$\sigma \in (0,1)$ such that, for all~$n\in \mathbb N$,
\begin{equation} \label{grading:parameter}
\hE \simeq \dist(\fc_K, K) \simeq\sigma ^{n-j} \quad \quad \forall \E \in L_n^j, \quad \forall j=1,\dots,n
\end{equation}
and
\begin{equation}
  \label{grading:parameter2}
 \hE \simeq \sigma^n \quad \quad \forall \E\in L^0_n.
\end{equation}
The conditions~\eqref{grading:parameter}--\eqref{grading:parameter2} asserts that the elements abutting the vertices in~$\Ncal$ are small, whereas the elements in the layers with large index~$j$ have fixed size asymptotically.
Note that the assumption (\textbf{A0-$\h\p$}) is satisfied automatically.
We require an additional assumption, which is necessary to show the exponential convergence result of Theorem \ref{theorem:hp-version} below; see~\cite[Assumption (\textbf{D4})]{hpVEMcorner}.
\begin{enumerate}
\item[(\textbf{A4}-$hp$)] For all $n\in \mathbb{N}$, let $\taun^1 = \taun
  \setminus L^0_n$. There exist a collection of
  squares $\cQ_n$ such that
  \begin{itemize}
  \item $\card(\cQ_n) = \card(\taun^1)$; for each $K\in \taun^1$, there exists
    $Q=Q(K)\in \cQ_n$ such that $K\subset Q$  and $h_K\simeq h_Q$. Additionally,
    $\dist(\fc_K, Q(K))\simeq h_K$;
  \item every $\xbf\in \Omega$ belongs at most to a fixed number of squares $Q$,
    uniformly in the discretization parameters.
  \end{itemize}
  In addition, for all $K\in L^0_n$, $K$ is star shaped with
  respect to $\fc_K$ and the subtriangulation obtained by joining $\fc_K$
  with the other vertices of $K$ is shape regular.
\end{enumerate}
Although necessary in the proof of Theorem~\ref{theorem:hp-version}, the condition (\textbf{A4}-$hp$) is not necessary in practice.
For instance, the $\h\p$-version of the method converges exponentially also on meshes, as those depicted in Figure \ref{figure:layers-Lshaped} (right); see Section~\ref{subsection:hp-version} below.

Next, we introduce a distribution of degrees of accuracy, by picking a high degree on large elements, where the solution is smooth, and decrease such degree linearly while decreasing the size of the elements.
More precisely, given a positive parameter~$\mu$,
set~$\Tn := \card(\taun)$ and introduce~$\pbf \in \mathbb N ^{\Tn}$ as follows:
\begin{equation} \label{distribution:elements}
\pbf_{\E} := \lceil \mu ({j}+1) \rceil \quad \quad  \text{where } \E \in L_n^j \quad \forall j=0,\dots,n+1.
\end{equation}
The vector~$\pbf$ represents the distribution of the degrees of accuracy over a mesh~$\taun$.
Given~$\Epsilonn := \card(\En)$, we also introduce a vector~$\pbfEn \in \mathbb N^{\Epsilonn}$, which represents the distribution of polynomial degrees over the skeleton of the mesh, and is defined as
\[
\pbfEn_{\e} := 
\begin{cases}
\max (\p_{\E_1}, \p_{\E_2}) 	& \text{if } \e \in \EnI \text{ and } \overline{\E_1} \cap \overline{\E_2}\\
\pE									& \text{if } \e \in \EnB \text{ and } \e \in \EE \text{ for some } \E \in \taun.
\end{cases}
\]
We can now define the $\h\p$-space for the velocities as the space of functions that are piecewise polynomials with distribution~$\pbfEn$ over the skeleton of the mesh
and which solve problems of the form~\eqref{local:problem} with right-hand side being polynomials of degree~$\pbf_{\E}-2$ (vector) and~$\p_{\E}-1$, respectively, on~$\E$.
On the other hand, we define the $\h\p$-virtual element space for the pressure as the space of piecewise polynomials of degree~$\pbf_{\E}$ on~$\E$.

Using the abstract analysis in Theorem~\ref{theorem:abstract} together with the tools in~\cite[Section~5]{hpVEMcorner}, we state the following result.
\begin{thm} \label{theorem:hp-version}
Let~$\{ \taun \}_{n \in \mathbb N}$ be a sequence of geometrically refined meshes satisfying the assumptions (\textbf{A1}), (\textbf{A2}), and (\textbf{A4}-$hp$),
with grading parameter~$\sigma$ satisfying~\eqref{grading:parameter} and \eqref{grading:parameter2}.
Let the  virtual element spaces $\Vn$ and $\Qn$ be constructed in an $\h\p$-fashion with suitable choice of the parameter~$\mu$ in~\eqref{distribution:elements}.
Suppose that
  there exist $c_\gamma>0$ and $k\in\mathbb{R}$ such that, for all $p\in
  \mathbb{N}$, $\gamma(p)\leq c_\gamma p^k$, with  $\gamma(p)$ defined in \eqref{gammap}.

Let the right-hand side~$\fbf$ be analytic in $\Omega$ and let
$(\ubf, s)$ and~$(\un, \sn)\in \Vn\times \Qn$ be the solutions
to~\eqref{Stokes:weak} and~\eqref{Stokes:VEM}, respectively.
For all~$n \in \mathbb N$, define~$\NV:= \card(\Vn) + \card(\Qn)$. 
Then, there exist two positive constants~$C$ and~$b$ such that
\[
\vert \ubf - \Pinablap \un \vert_{1,\taun} + \Vert s -\sn \Vert_{0,\Omega} \le  C \exp (-b \sqrt[3] \NV).
\]
\end{thm}

\begin{proof}
Starting from the abstract error analysis in Theorem~\ref{theorem:abstract}, it suffices to be able to show $\h\p$-upper bounds on the four terms appearing on the right-hand side of~\eqref{eq:apriori-up}.
More precisely,
from \cite[Lemmata~2 and 3]{hpVEMcorner}, there exist constants~$C_1$ and~$b_1>0$ such that, for all~$n\in \mathbb{N}$,
\[
\inf_{\upi\in [\mathbb{P}_{\p}(\taun)]^2}|\ubf-\upi|_{1,\taun}  \leq C_1 \exp(-b_1 n)
\]
Furthermore, noting that the pressure $s$ has the same regularity as the components of the gradient of the velocity $\ubf$, with similar arguments, we deduce that there exist constants~$C_2$ and~$b_2>0$ such that, for all~$n\in \mathbb{N}$,
\[
\inf_{\qn \in \Qn}\|s-\qn\|_{0, \Omega}  \leq C_2 \exp(-b_2 n).
\]
Then, we deduce from~\cite[Lemmata 4 and 5]{hpVEMcorner} that there exist constants~$C_3$ and~$b_3>0$ such that, for all~$n\in \mathbb{N}$,
\[
    \Fcaln \leq C_3\exp(-b_3 n).
\]
Finally, the estimate 
\begin{equation*}
\inf_{\vntilde\in \Vntilde}|\ubf-\vntilde|_{1, \Omega}  \leq C_4 \exp(-b_4 n),
\end{equation*}
for constants $C_4, b_4>0$, independent of $n$ is a consequence of~\cite[Lemmata 6 and 7]{hpVEMcorner}.

We remark that \eqref{grading:parameter}, \eqref{grading:parameter2}, and~\eqref{distribution:elements} imply~$\card(\taun)\simeq n$, see, e.g., \cite[Equation (5.6)]{Hiptmair2014}.
Since $\dim(\VnE)\simeq p_{\E}^2$ and $\dim(\mathbb{P}_{p_\E}(\E))\simeq p_\E^2$ for each $\E\in\taun$, \eqref{distribution:elements} gives~$N_V\simeq n^3$.
Since $\gamma(p)$ grows at most algebraically in terms of~$\p$, we  absorb the term~$\gamma(p)^{\frac{3}{2}}$ appearing on the right-hand side constants.
This concludes the proof.
\end{proof}

The assumption in Theorem~\ref{theorem:hp-version} that~$\gamma(\p)$ grows at most algebraically in terms of~$\p$
is fulfilled, e.g., by the stabilization introduced in Remark~\ref{remark:explicit-STAB}.

%%%%%%%%%%%%%%%%%%%%%%%%%%%%%%%%%%%%%%%%%%%%%%%%%%%%%%%%%%%%%%%%%%%%%%%%%%%
\section{Numerical results} \label{section:numerical-experiments}
%%%%%%%%%%%%%%%%%%%%%%%%%%%%%%%%%%%%%%%%%%%%%%%%%%%%%%%%%%%%%%%%%%%%%%%%%%%
In this section, we present numerical results which validate the theoretical predictions of Theorems~\ref{theorem:p-version} and~\ref{theorem:p-version}: see Sections~\ref{subsection:p-version} and~\ref{subsection:hp-version}, respectively.

We perform the numerical experiments on the two following test cases.
\subsubsection*{Test case~1. }
Given~$\Omega_1 := (0,1)^2$, we consider the analytic solution
\begin{equation} \label{solution1}
\ubf_1 := 
\begin{pmatrix}
-0.5 \cos^2(\pi\, x - \frac{\pi}{2}) \cos(\pi\, y) \sin (\pi y)\\
 0.5 \cos^2(\pi\, y - \frac{\pi}{2}) \cos(\pi \, x) \sin (\pi \, x)
\end{pmatrix} , \quad \quad s _1:= \sin(\pi\, x) - \sin(\pi\, y).
\end{equation}
The boundary conditions of the velocity are homogeneous on the whole boundary. The right-hand side~$\fbf$ is computed accordingly.

\subsubsection*{Test case~2. }
As a second test case, we consider a singular function on the L-shaped domain~$\Omega_2 := (-1,1)^2 \setminus [0,1) \times (-1,0]$.
Let
\begin{equation} \label{alpha-omega}
\omega := 3\pi/2 , \quad \quad \alpha = 0.54448373678246\dots
\end{equation}
Note that $\alpha$ is the smallest positive solution to equation
\eqref{eq:sing-exp}, with $\fc = (0,0)$ and $\phi_{\fc} = \omega$.
Given~$(r,\theta)$ the polar coordinates at the re-entrant corner~$(0,0)$,
introduce the auxiliary function
\[
\psi(r,\theta) =\frac{\sin((1+\alpha) \theta) \cos(\alpha \omega) }{ 1+\alpha} - \cos((1+\alpha) \theta)
- \frac{ \sin((1-\alpha) \theta) \cos(\alpha \omega)  } {1-\alpha } + \cos((1-\alpha) \theta).
\]
The singular solution we approximate is
\begin{equation} \label{solution2}
\ubf_2 := 
\begin{pmatrix}
r^{\alpha} \left(  (1+\alpha) \sin(\theta) \psi (\theta) + \cos(\theta) \psi' (\theta)    \right)\\
r^{\alpha} \left(   \sin(\theta) \psi' (\theta)  - (1+\alpha) \cos(\theta) \psi (\theta)   \right)
\end{pmatrix} ,
\quad s _2:= r^{\alpha-1} \left(  (1+\alpha)^2 \psi'(\theta) +  \psi^{(3)} (\theta)     \right) / (1-\alpha).
\end{equation}
This solution is such that the Stokes equation is homogeneous, i.e., $\fbf =0$.
Moreover, the Dirichlet conditions are homogeneous along the edges abutting the re-entrant corner.

\subsubsection*{Meshes. }
We are interested in the $\p$- and $\h\p$-versions of the method.
The specific construction of the mesh is not central to the convergence
  properties of the $\p$-version. Therefore, we only employ uniform
Cartesian meshes both on the square domain~$\Omega_1$ and on the L-shaped domain~$\Omega_2$.
As for the meshes to employ for the $\h\p$-version, we postpone their construction to Section~\ref{subsection:hp-version} below.

\subsubsection*{Stabilization.}
In Remark~\ref{remark:explicit-STAB}, we introduced a stabilization with explicit bounds in~\eqref{eq:stab-constants} in terms of the degree of accuracy~$\p$.
Notwithstanding, in the forthcoming numerical experiments, we resort to the so-called D-recipe, see~\cite{bricksVEM}.
Given~$\E \in \taun$, introduce the local canonical basis~$\{\varphibold_j\}_{j=1}^{\dim(\VnE)}$ of the space~$\VnE$,
which is dual to the degrees of freedom~$\{\dofbf _j(\cdot)\}_{j=1}^{\dim(\VnE)}$ introduced in Section~\ref{subsection:VES-Stokes}.
We define
\[
\SED (\un,\vn) := \sum_{j=1}^{\dim(\VnE)} \max(1, \vert \Pinablap \varphibold_j \vert_{1,\E}) \dofbf_j(\un) \dofbf_j(\vn).
\]
It is known~\cite{fetishVEM, VEM3Dbasic} that stabilizations of this sort lead to effective performance of the method.

We highlight that we also tested the method with the stabilization~\eqref{eq:SEexplicit}, and this leads to results that are comparable to those that we present in the forthcoming sections.

\subsubsection*{Polynomial bases.}
We refer to~\cite{bricksVEM}, as for the choice of the polynomial bases.
We underline that this choice could be improved; see Remark~\ref{remark-ill-cond} below.

\subsubsection*{Errors.}
We are interested in the convergence rate of the two following quantities:
\[
\vert \ubf - \Pinablap \un \vert_{1,\taun} , \quad \quad \Vert s - \sn \Vert_{0,\E}.
\]
Indeed, Theorems~\ref{theorem:p-version} and~\ref{theorem:hp-version} provide upper bounds on such two quantities.

%%%%%%%%%%%%%%%%%%%%%%%
\subsection{The $\p$-version of the method} \label{subsection:p-version}
%%%%%%%%%%%%%%%%%%%%%%%
In this section, we present numerical results validating the theoretical predictions of Theorem~\ref{theorem:p-version} for the $\p$-version of the method.
We consider the exact solutions~$(\ubf_1, s_1)$ and~$(\ubf_2, s_2)$ in~\eqref{solution1} and~\eqref{solution2}, respectively.
We employ a coarse mesh of $2\times 2$ uniform squares on the domain~$\Omega_1$.

\begin{figure}  [h]
\centering
\includegraphics [angle=0, width=0.45\textwidth]{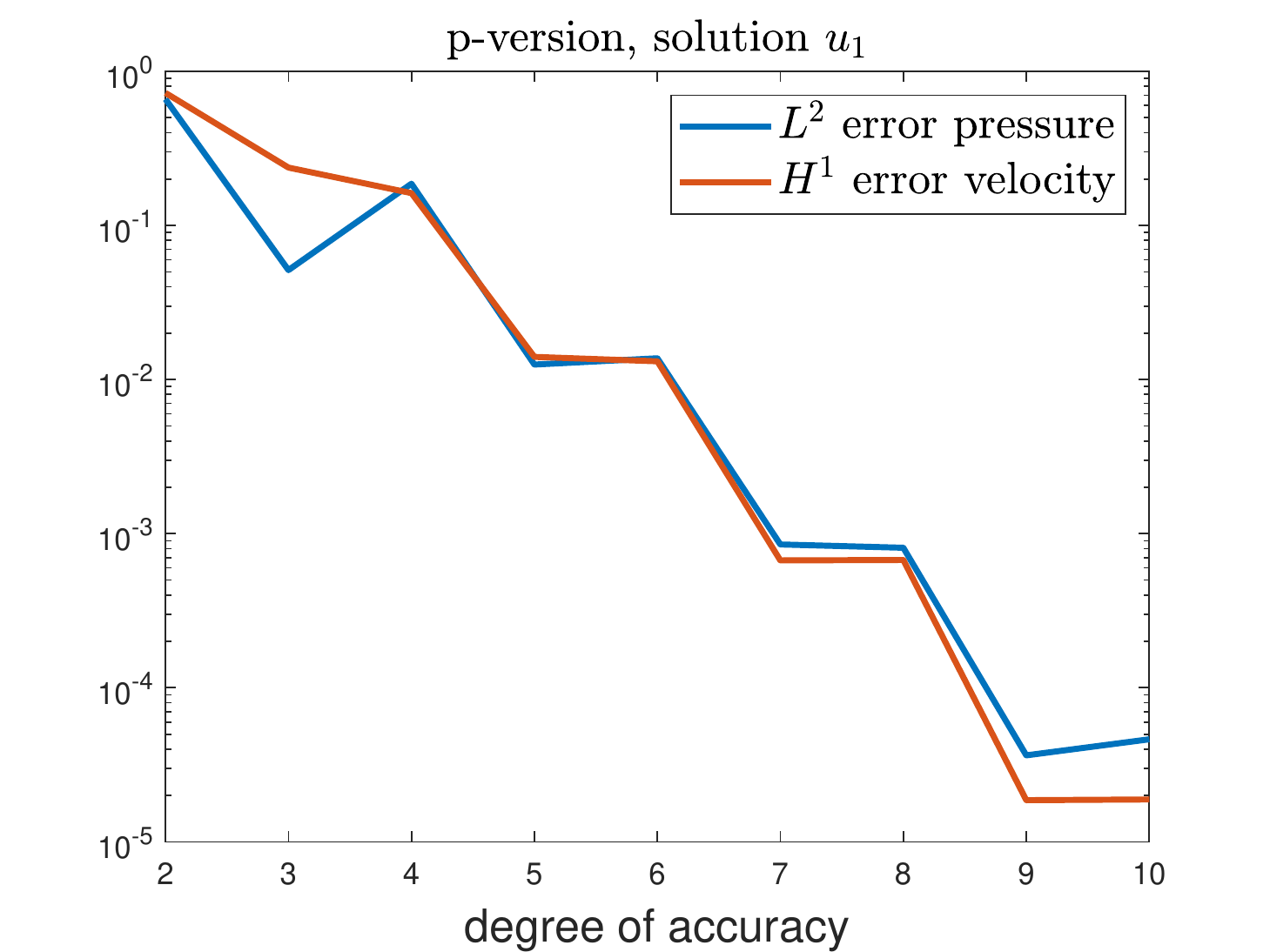}
\includegraphics [angle=0, width=0.45\textwidth]{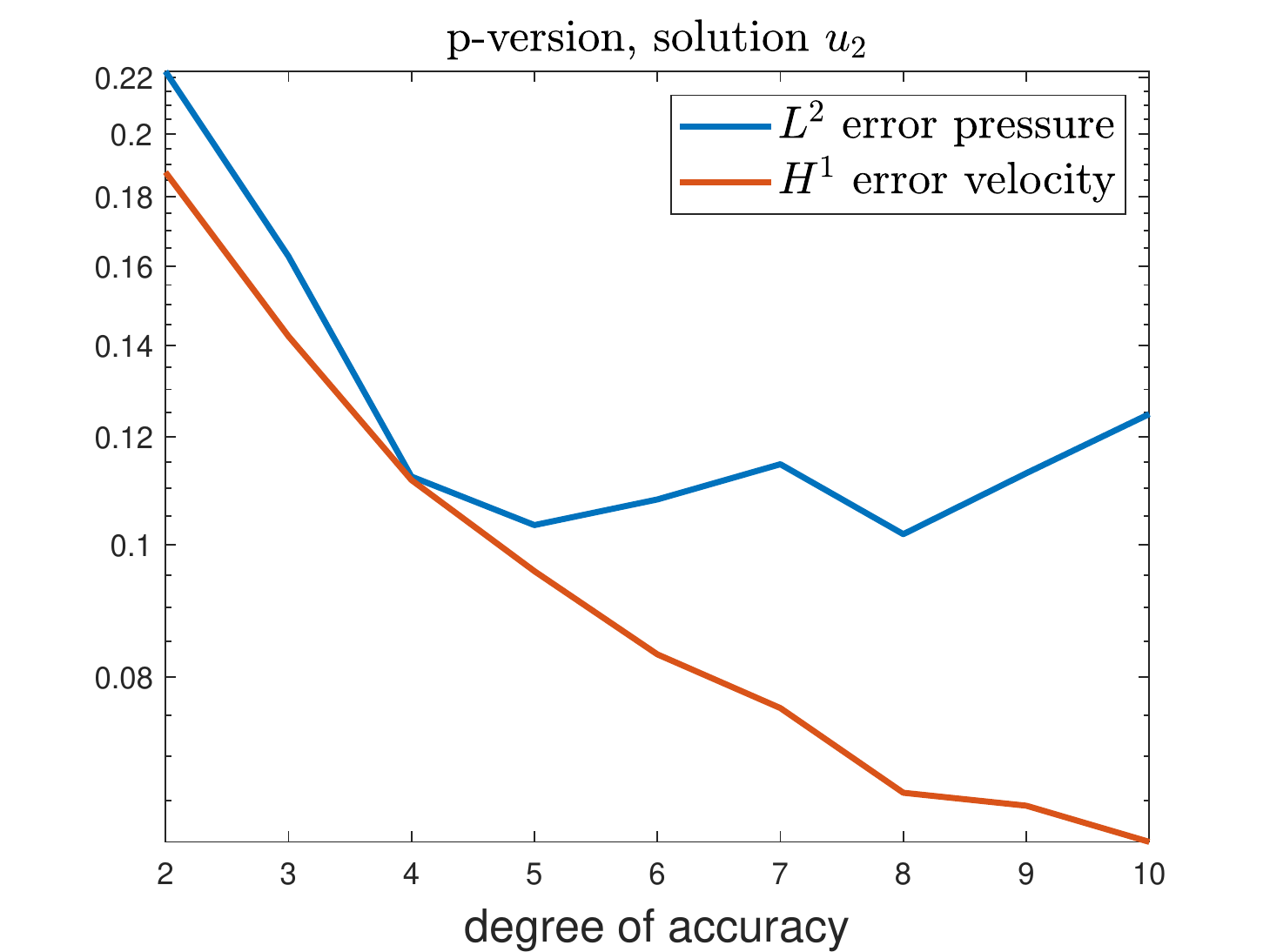}
\caption{$\p$-version of the method. We consider the exact solutions~$(\ubf_1, s_1)$ and~$(\ubf_2, s_2)$ defined in~\eqref{solution1} and~\eqref{solution2} in the left and right panel, respectively.
We plot the errors~$\Vert s_j - \sn \Vert_{0,\Omega}$ and~$\vert \ubf_j - \un \vert_{1,\taun}$, for~$j=1,2$.
We employ a coarse mesh of~$2 \times 2$ uniform squares.}
\label{figure:p-version}
\end{figure}

As expected from the theoretical predictions, in Figure~\ref{figure:p-version},
we observe exponential convergence for the test case with smooth solution, and only algebraic convergence for the singular solution case.

\begin{remark} \label{remark-ill-cond}
For the exact solution~$\ubf _2$, the~$L^2$ error on the pressure stagnates at around~$\p=4$ and then grows.
Similarly, the~$H^1$ error stagnates starting from~$\p=9$.
This behaviour can be traced back to the ill-conditioning of the  resulting  linear system, which is mainly due to the choice of the polynomial bases in the definition of the degrees of freedom
and in the expansion of the polynomial projectors.
A possible remedy to this problem might be an orthogonalization process of the polynomial bases; see, e.g., \cite{fetishVEM}.
For the sake of clarity, we avoid such investigation here.
\end{remark}

%%%%%%%%%%%%%%%%%%%%%%%
\subsection{The $\h\p$-version of the method} \label{subsection:hp-version}
%%%%%%%%%%%%%%%%%%%%%%%
As predicted in Theorem~\ref{theorem:p-version} and observed in Figure~\ref{figure:p-version} numerically,
the method converges in terms of the degree of accuracy~$\p$ algebraically, whenever the exact solution is not analytic.
However, as discussed in Theorems~\ref{theorem:regularity} and~\ref{theorem:regularity2},
solutions to the Stokes problem~\eqref{Stokes:weak} on polygonal domains with smooth data
belong to the Kondrat'ev spaces~$\cK^{\varpi}_\ugamma(\Omega) $ in~\eqref{eq:cKanalytic}.
In general, for solutions~$(\ubf, s)$ to the Stokes problem in a nonconvex domain $\Omega$, we can expect $\ubf\in \left[H^k(\Omega)  \right]^2$ and $s \in H^{k-1}(\Omega)$ for a given~$k<2$ only.

Exponential convergence can be recovered for weighted analytic functions, by employing $\h\p$-approximation spaces,
following the gospel of Babu\v ska and collaborators, as proven in Theorem~\ref{theorem:hp-version}.
See also, e.g., \cite{BabuGuo_hpFEM, babuskaguo_curvilinearhpFEM, SchwabpandhpFEM} and the references therein.

Thus, in this section, we validate the theoretical predictions of Theorem~\ref{theorem:hp-version}. To this aim, we consider the test case with exact solution~$(\ubf_2, s_2)$ in~\eqref{solution2}.
We construct the distribution of the degrees of accuracy by picking~$\mu=1$ in~\eqref{distribution:elements}.
Moreover, we employ $\h\p$-virtual element spaces based on geometric meshes as those depicted in Figure~\ref{figure:layers-Lshaped}.
There, we depict meshes with three layers, which are geometrically refined towards the re-entrant corner~$(0,0)$ in three different ways.
The numbers within the elements represent the local degrees of accuracy of the method.

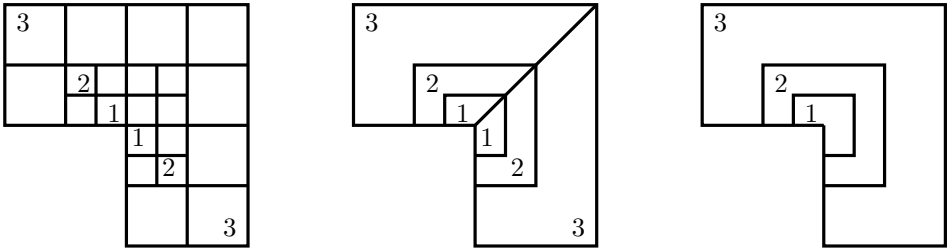
\begin{figure}  [h]
\begin{center} 
  \begin{minipage}{0.30\textwidth}
    \begin{tikzpicture}[scale=0.4]
     \draw[black, very thick, -] (0,0) -- (0,-4) -- (4,-4) -- (4,4) -- (-4,4) -- (-4,0) -- (0,0);
     \draw[black, very thick, -] (-2,0) -- (-2,4); \draw[black, very thick, -] (0,0) -- (0,4); \draw[black, very thick, -] (2,-4) -- (2,4);
	\draw[black, very thick, -] (-1,0) -- (-1,2); \draw[black, very thick, -] (1,-2) -- (1,2);
	\draw[black, very thick, -] (0,-2) -- (4,-2);  \draw[black, very thick, -] (0,0) -- (4,0); \draw[black, very thick, -] (-4,2) -- (4,2); 
	\draw[black, very thick, -] (-2,1) -- (2,1); \draw[black, very thick, -] (0,-1) -- (2,-1);
\draw(-3.4,3.4) node[black] {{$3$}}; \draw(3.4, -3.4) node[black] {{$3$}};
\draw(-1.4,1.4) node[black] {{$2$}}; \draw(1.4, -1.4) node[black] {{$2$}};
\draw(-.4,.4) node[black] {{$1$}}; \draw(.4, -.4) node[black] {{$1$}};
    \end{tikzpicture}
  \end{minipage}
  \begin{minipage}{0.30\textwidth}
    \begin{tikzpicture}[scale=0.4]
     \draw[black, very thick, -] (0,0) -- (0,-4) -- (4,-4) -- (4,4) -- (-4,4) -- (-4,0) -- (0,0);
      \draw[black, very thick, -] (0,-2) -- (2,-2) -- (2,2) -- (-2,2) -- (-2,0);
      \draw[black, very thick, -] (0,-1) -- (1,-1) -- (1,1) -- (-1,1) -- (-1,0);
      \draw[black, very thick, -] (0,0) -- (4,4);
      \draw(-3.4,3.4) node[black] {{$3$}}; \draw(3.4, -3.4) node[black] {{$3$}};
      \draw(-1.4,1.4) node[black] {{$2$}}; \draw(1.4, -1.4) node[black] {{$2$}};
      \draw(-.4,.4) node[black] {{$1$}}; \draw(.4, -.4) node[black] {{$1$}};
    \end{tikzpicture}
  \end{minipage}
  \begin{minipage}{0.33\textwidth}
    \begin{tikzpicture}[scale=0.4]
     \draw[black, very thick, -] (0,0) -- (0,-4) -- (4,-4) -- (4,4) -- (-4,4) -- (-4,0) -- (0,0);
      \draw[black, very thick, -] (0,-2) -- (2,-2) -- (2,2) -- (-2,2) -- (-2,0);
      \draw[black, very thick, -] (0,-1) -- (1,-1) -- (1,1) -- (-1,1) -- (-1,0);
      \draw(-3.4,3.4) node[black] {{$3$}}; 
      \draw(-1.4,1.4) node[black] {{$2$}}; 
      \draw(-.4,.4) node[black] {{$1$}}; 
    \end{tikzpicture}
  \end{minipage}
\end{center}
\caption{Examples of meshes that are geometrically refined towards the
  re-entrant corner~$(0,0)$. Here, the grading parameter~$\sigma$
  satisfying~\eqref{grading:parameter} and~\eqref{grading:parameter2} is~$1/2$.
The numbers in the elements denote the local degree of accuracy. In particular, we have picked~$\mu=1$ in~\eqref{distribution:elements}.}
\label{figure:layers-Lshaped}
\end{figure}

In Figures~\ref{figure:hp-version:BabuSchwab}, \ref{figure:hp-version:Chernov-cut}, and~\ref{figure:hp-version:Chernov-no-cut},
we depict the decay of the errors in~\eqref{computable:errors} employing $\h\p$-virtual element spaces based on meshes as those in Figure~\ref{figure:layers-Lshaped}.
We pick different choices of the grading parameter~$\sigma$.

\begin{figure}  [h]
\centering
\includegraphics [angle=0, width=0.45\textwidth]{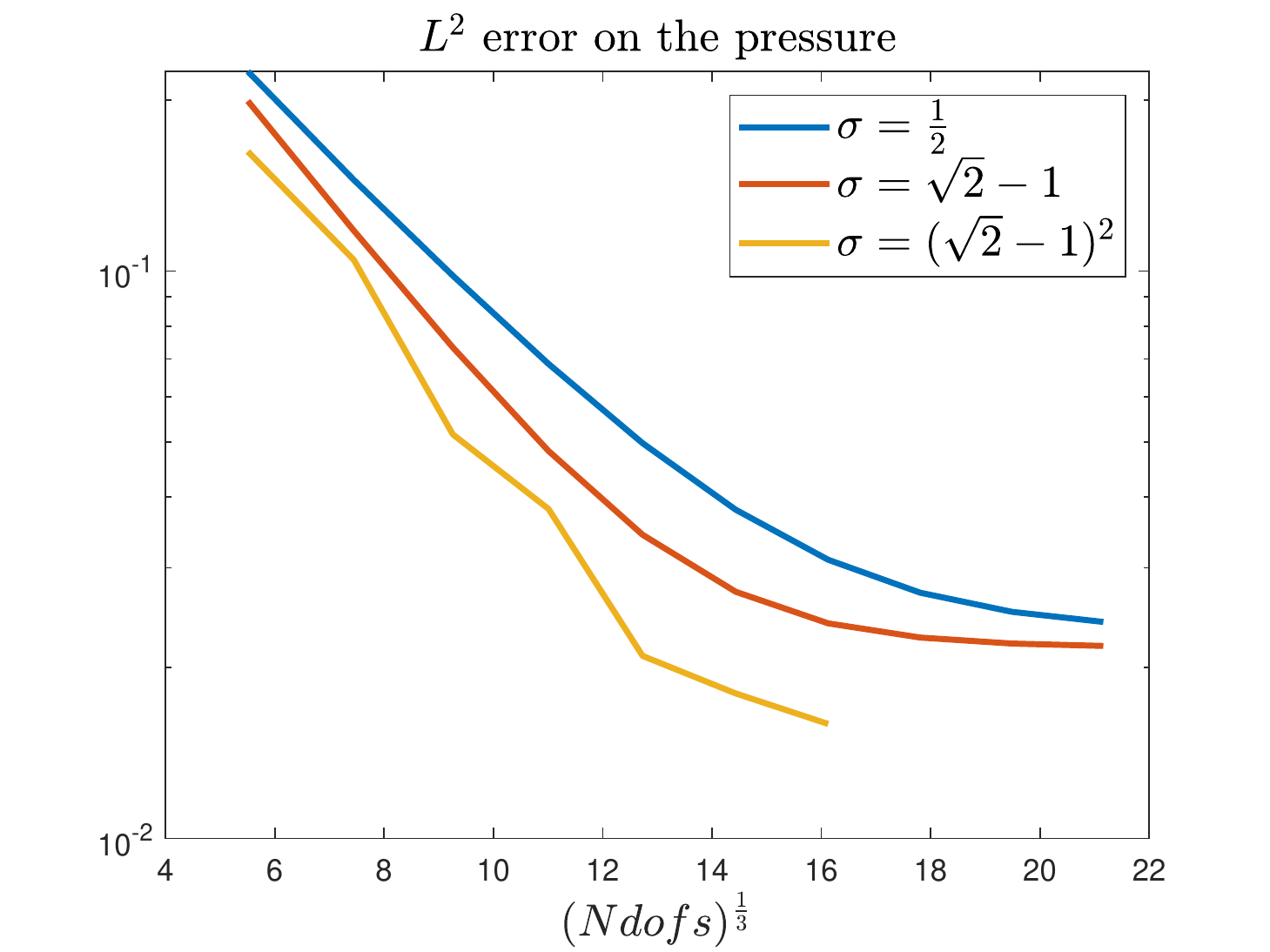}
\includegraphics [angle=0, width=0.45\textwidth]{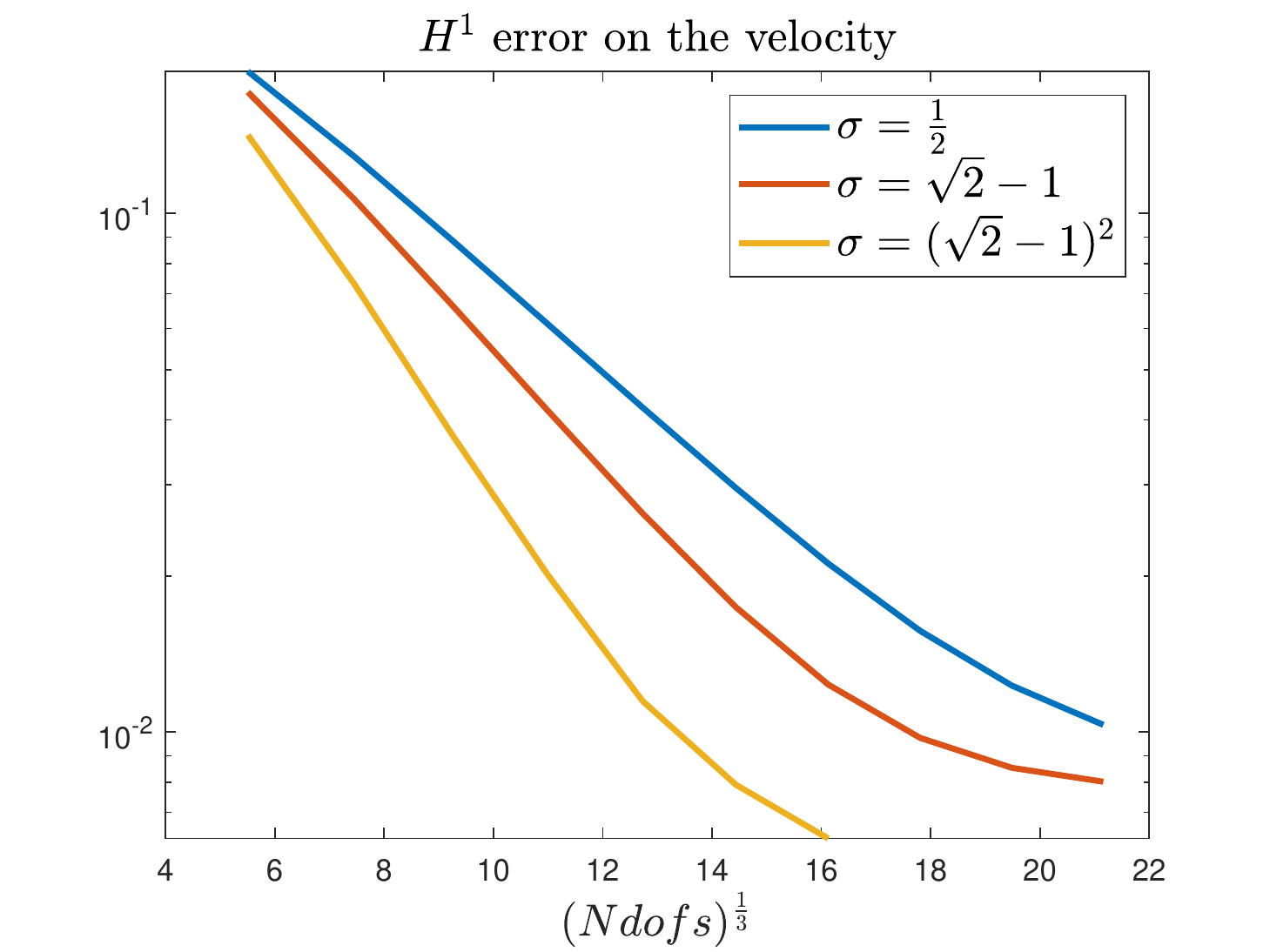}
\caption{$\h\p$-version of the method. We consider the exact solution~$(\ubf_2, s_2)$ defined in~\eqref{solution2}.
Left panel: $\Vert s_j - \sn \Vert_{0,\Omega}$, $j=1,2$. Right panel: $\vert \ubf_j - \un \vert_{1,\taun}$, $j=1,2$.
We employ meshes that are geometrically refined towards the re-entrant corner as those in Figure~\ref{figure:layers-Lshaped} (left).
We pick three different choices of the parameter~$\sigma$ satisfying~\eqref{grading:parameter} and~\eqref{grading:parameter2}, namely $\sigma=\frac{1}{2}$, $\sigma=\sqrt 2 - 1$, and~$\sigma=(\sqrt 2 -1)^2$.}
\label{figure:hp-version:BabuSchwab}
\end{figure}

\begin{figure}  [h]
\centering
\includegraphics [angle=0, width=0.45\textwidth]{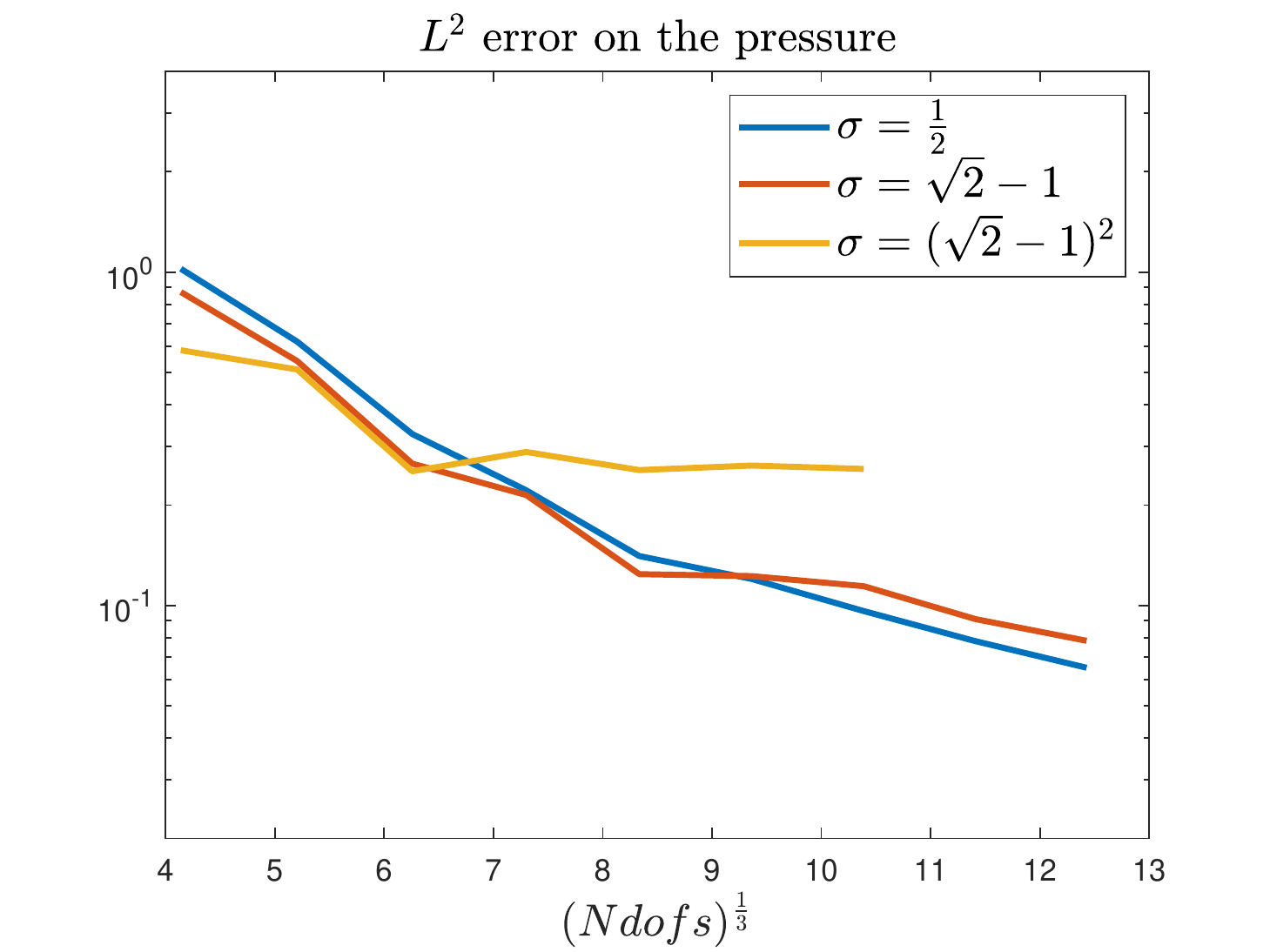}
\includegraphics [angle=0, width=0.45\textwidth]{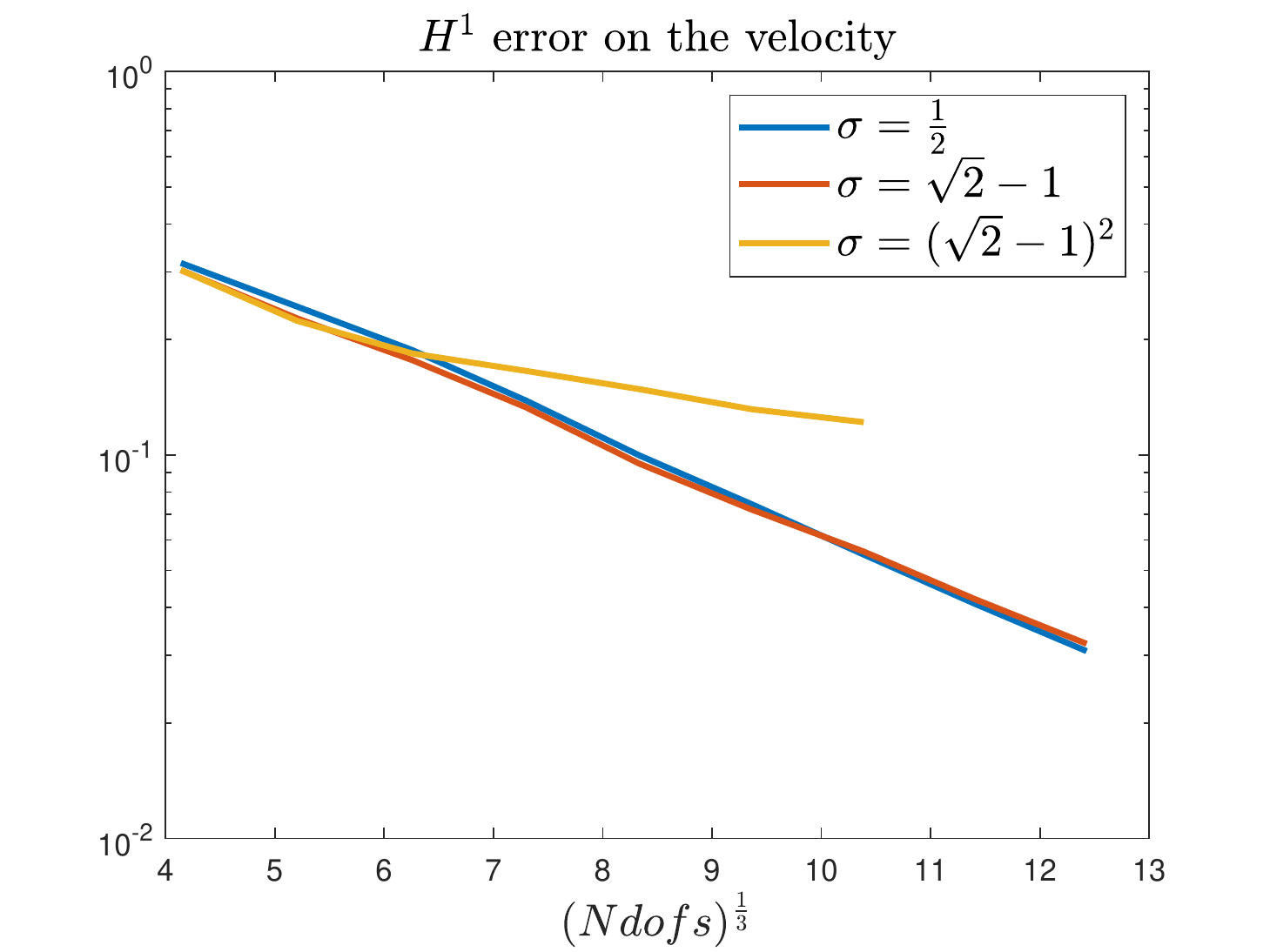}
\caption{$\h\p$-version of the method. We consider the exact solution~$(\ubf_2, s_2)$ defined in~\eqref{solution2}.
We employ meshes that are geometrically refined towards the re-entrant corner as those in Figure~\ref{figure:layers-Lshaped} (centre).
We pick three different choices of the parameter~$\sigma$ satisfying~\eqref{grading:parameter} and~\eqref{grading:parameter2}, namely $\sigma=\frac{1}{2}$, $\sigma=\sqrt 2 - 1$, and~$\sigma=(\sqrt 2 -1)^2$.}
\label{figure:hp-version:Chernov-cut}
\end{figure}

\begin{figure}  [h]
\centering
\includegraphics [angle=0, width=0.45\textwidth]{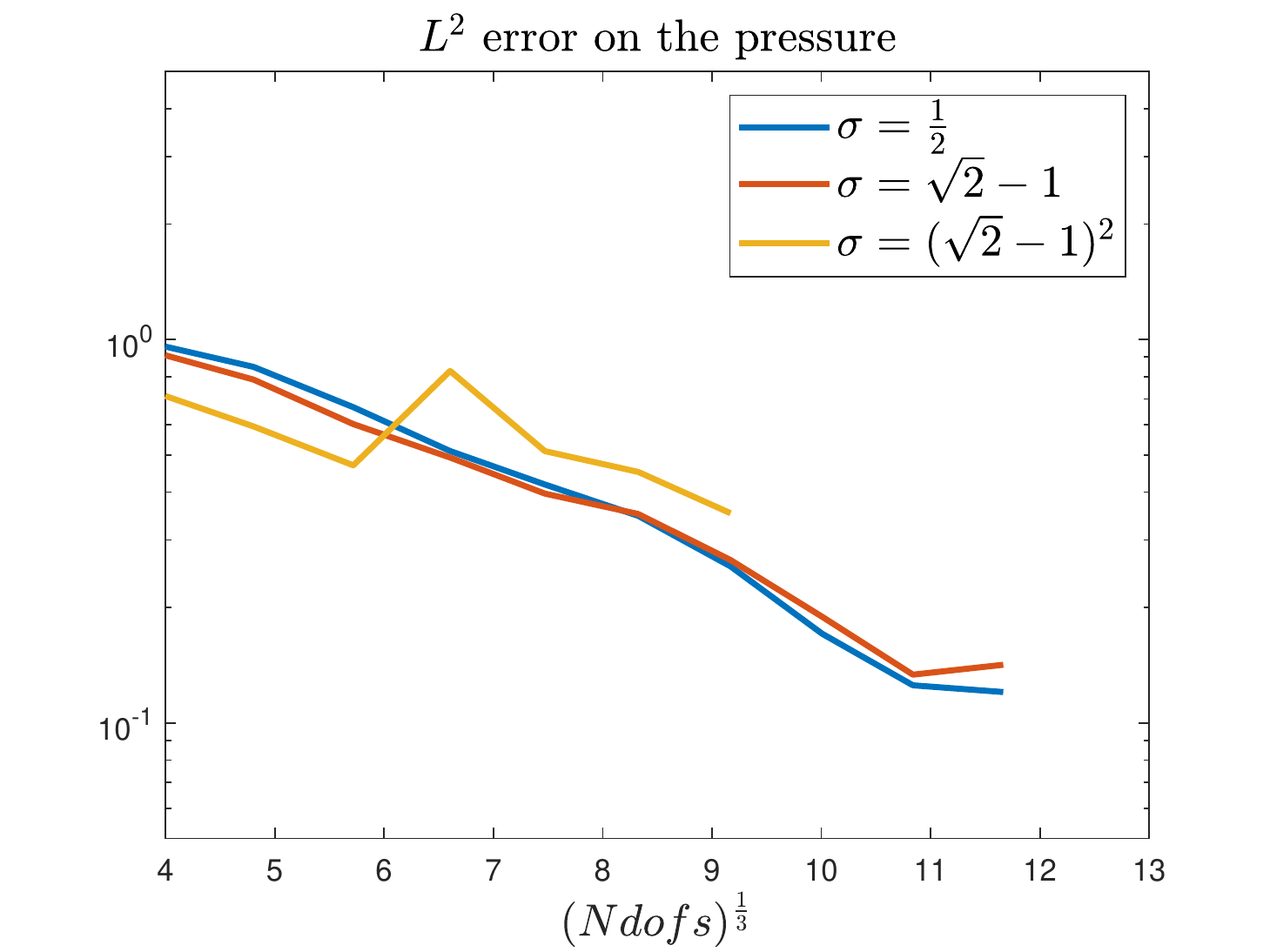}
\includegraphics [angle=0, width=0.45\textwidth]{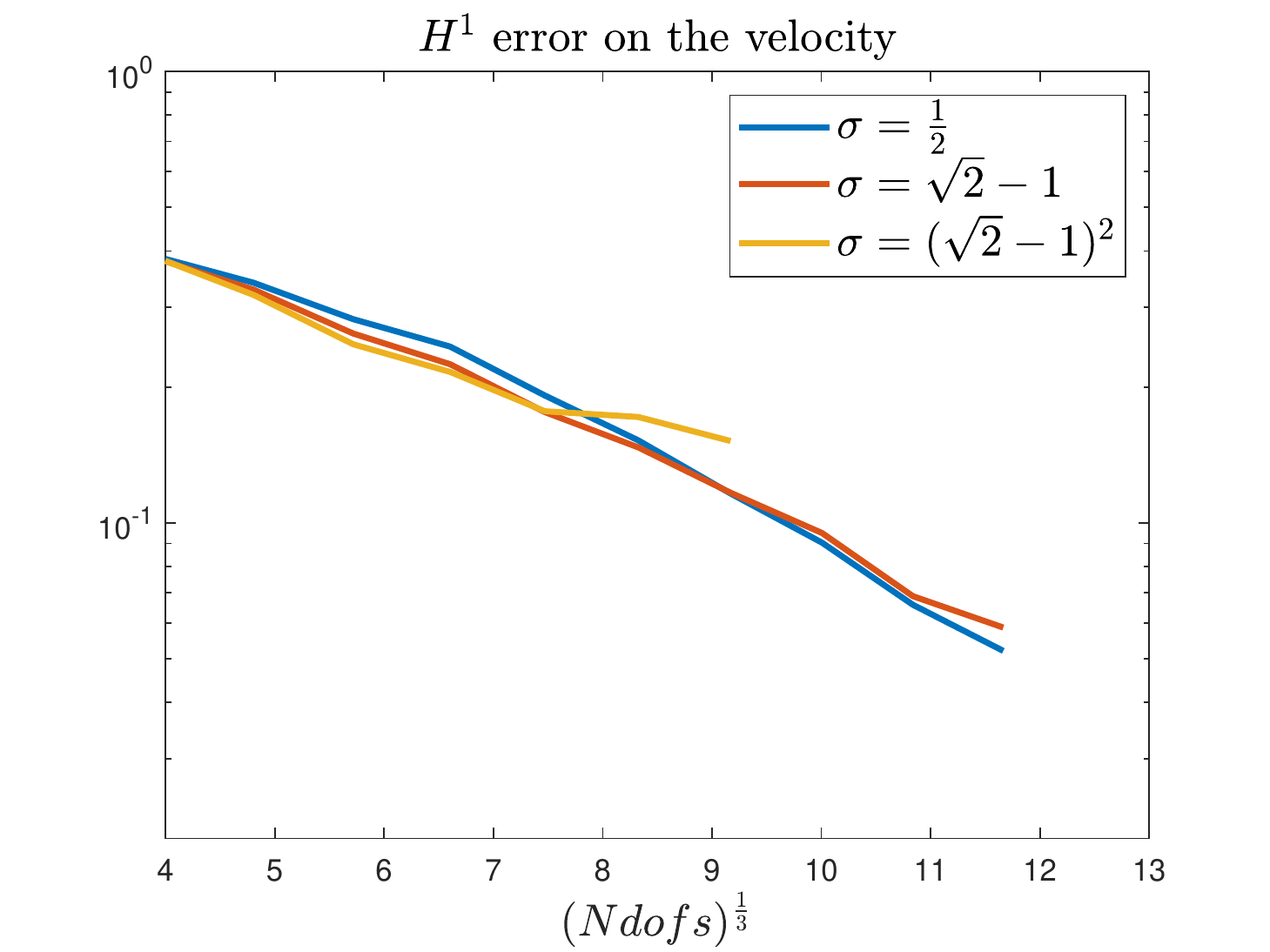}
\caption{$\h\p$-version of the method. We consider the exact solution~$(\ubf_2, s_2)$ defined in~\eqref{solution2}.
Left panel: $\Vert s_j - \sn \Vert_{0,\Omega}$, $j=1,2$. Right panel: $\vert \ubf_j - \un \vert_{1,\taun}$, $j=1,2$.
We employ meshes that are geometrically refined towards the re-entrant corner as those in Figure~\ref{figure:layers-Lshaped} (right).
We pick three different choices of the parameter~$\sigma$ satisfying~\eqref{grading:parameter} and~\eqref{grading:parameter2}, namely $\sigma=\frac{1}{2}$, $\sigma=\sqrt 2 - 1$, and~$\sigma=(\sqrt 2 -1)^2$.} 
\label{figure:hp-version:Chernov-no-cut}
\end{figure}

We observe exponential decay of the errors.
The error saturation due to ill-conditioning manifests itself earlier for some of the meshes depicted in Figure~\ref{figure:layers-Lshaped}.
See Remark~\ref{remark-ill-cond} for further details on this point.

%%%%%%%%%%%%%%%%%%%%%%%%%%%%%%%%%%%%%%%%%%%%%%%%%%%%%%%%%%%%%%%%%%%%%%%%%%%
\section{Conclusions} \label{section:conclusions}
%%%%%%%%%%%%%%%%%%%%%%%%%%%%%%%%%%%%%%%%%%%%%%%%%%%%%%%%%%%%%%%%%%%%%%%%%%%
We have analysed the $\p$- and $\h\p$-versions of the virtual element method for a 2D Stokes problem on polygonal domains.
In particular, we have shown that the $\h\p$-VEM converges with exponential rate to the solution of Stokes problems in polygonal domains, with smooth right-hand side.
In addition, we have proven algebraic and exponential convergence rate of
the $\p$-version of the method for solutions with (sufficiently high) finite
  Sobolev regularity and for analytic solutions, respectively.
The novel technical tool we introduced in this work is the proof of the existence of a bijection operator between Poisson-like and Stokes-like virtual element spaces for the velocity.
This allows us to leverage known results from the analysis of the Poisson problem in a straightforward manner.
The numerical experiments we performed validate and extend the theoretical results.
Future investigations will cover the analysis of $\p$- and $\h\p$-VEM for the Navier-Stokes equation and three dimensional problems.

%%%%%%%%%%%%%%%%%%%%%%%%%%%%%%%%%%%%%%%%%%%%%%%%%%%%%%%%%%%%%%%%%%%%%%%%%%%
\bibliographystyle{plain}
{\footnotesize
\bibliography{bibliogr}
}

%%%%%%%%%%%%%%%%%%%%%%%%%%%%%%%%%%%%%%%%%%%%%%%%%%%%%%%%%%%%%%%%%%%%%%%%%%%

\end{document}